\documentclass[12pt]{amsart}
\usepackage{amsmath,amsfonts,amssymb,amsthm}
\usepackage{graphics}

\numberwithin{equation}{section}

\newtheoremstyle{thm}
  {9pt}{9pt}{\itshape}{}{\bfseries}{}{.5em}{}
\theoremstyle{thm}
\newtheorem{thm}{Theorem}[section]
\newtheorem{cor}[thm]{Corollary}
\newtheorem{lemma}[thm]{Lemma}

\newtheoremstyle{defin}
  {9pt}{9pt}{}{}{\bfseries}{}{.5em}{}
\theoremstyle{defin}

\newtheoremstyle{exm}
  {9pt}{9pt}{}{}{\scshape}{}{.5em}{}
\theoremstyle{exm}
\newtheorem{exm}[thm]{Example}
\newtheorem{rmk}[thm]{Remark}
\newtheoremstyle{proof}
  {}{}{}{}{\itshape}{:}{.5em}{}
\theoremstyle{proof}

\newtheorem*{skt}{Sketch of proof}

\newcommand{\set}[1]{\{#1\}}

\newcommand{\C}{{\mathbb C}}
\newcommand{\p}[1]{\mathcal{#1}}
\newcommand{\f}[1]{\mathfrak{#1}}
\newcommand{\un}[1]{\underline{#1}}
\newcommand{\wT}{{\widetilde T}}
\newcommand{\wP}{{\widetilde P}}
\DeclareMathOperator{\tr}{tr}
\DeclareMathOperator{\rw}{rw}

\DeclareMathOperator{\sign}{sign}

\DeclareMathOperator{\weight}{weight}
\DeclareMathOperator{\hght}{ht}
\DeclareMathOperator{\wdth}{wt}
\DeclareMathOperator{\strp}{st}
\DeclareMathOperator{\End}{End}

\author{Matja\v z Konvalinka}
\title[Characters of the Hecke algebra]{On combinatorial formulas for the characters of Hecke algebras}

\begin{document}

\begin{abstract}
 Hecke algebras are beautiful $q$-extensions of Coxeter groups. In this paper, we prove several results on their characters, with an emphasis on characters induced from trivial and sign representations of parabolic subalgebras. While most of the results in type A are known, our proofs are of a combinatorial nature, and generalize to (partial) results in types B and C. We also present complete descriptions of such characters for type I.
\end{abstract}

\maketitle

\section{Introduction} \label{intro}

Let us start with a description of three families of characters of the symmetric group $\f S_n$. For a partition $\lambda \vdash n$, denote the character induced from the trivial character of the Young subgroups $\f S_{\lambda_1} \times \cdots \times \f S_{\lambda_p}$ of $\f S_n$ by $\eta_\lambda$. It is given by
\begin{equation} \label{intro1}
 \eta_\lambda(\pi) = R_{\mu \lambda},
\end{equation}
where
\begin{itemize}
 \item $\mu = (\mu_1,\ldots,\mu_r)$ is the type of the permutation $\pi$ (sequence of lengths of cycles of $\pi$), and
 \item $R_{\mu \lambda}$ is the number of ordered partitions $(B_1,\ldots,B_p)$ of the set $\set{1,\ldots,r}$ such that
 $$\lambda_j = \sum_{i \in B_j} \mu_i \quad \mbox{for} \quad 1 \leq j \leq p.$$
\end{itemize}

\medskip

The sign characters of Young subgroups induce the characters $\set{ \epsilon_\lambda : \lambda \vdash n}$, which are given by
\begin{equation} \label{intro3}
 \epsilon_\lambda(\pi) = \sigma_\mu R_{\mu \lambda},
\end{equation}
where $\sigma_\mu = \sign \pi = (-1)^{j_2+j_4+\ldots}$ for $\mu=\langle 1^{j_1}2^{j_2}\cdots \rangle$. See \cite[\S 7]{StanEC}.

\medskip

Finally, the irreducible characters of the symmetric group have the following combinatorial interpretation. A \emph{border strip} is a connected skew shape with no $2 \times 2$ square. Equivalently, a skew shape $\lambda/\mu$ is a border strip if and only if $\lambda_i = \mu_{i-1} + 1$ for $i \geq 2$. The height $\hght T$ of a border strip $T$ is one less than the number of rows, and the width $\wdth T$ is one less than the number of columns. A \emph{border strip tableau} of shape $\lambda/\mu$ and type $\alpha=(\alpha_1,\ldots,\alpha_p)$ is an assignment of positive integers to the squares of $\lambda/\mu$ such that:
\begin{itemize}
 \item every row and column is weakly increasing,
 \item the integer $i$ appears $\alpha_i$ times, and
 \item the set of squares occupied by $i$ forms a border strip or is empty.
\end{itemize}
The height $\hght$ of a border strip tableau $\p T$ is the sum of the heights of non-empty border strips that appear in $\p T$, and the width $\wdth$ is the sum of the widths of non-empty border strips that appear in $\p T$.

\medskip

For a partition $\lambda$, the irreducible character $\chi_\lambda$ is given by
\begin{equation} \label{intro2}
 \chi_\lambda(\pi) = \sum_{\p T} (-1)^{\hght \p T},
\end{equation}
where $\pi$ is a permutation in $S_n$ of type $\mu$, and $\p T$ runs over all border strip tableaux of shape $\lambda$ and type $\mu$.

\begin{exm} \label{intro4}
 Take $\lambda=(3,2,1)$ and $\pi = 214356$. Then $\eta_\lambda(\pi)=4$, corresponding to ordered partitions
 $$(\set{1,3},\set{2},\set{4}), \qquad (\set{1,4},\set{2},\set{3}),$$
 $$(\set{2,3},\set{1},\set{4}), \qquad (\set{2,4},\set{1},\set{3}).$$
 Since $\sign \pi = 1$, we have $\epsilon_\lambda(\pi) = 4$. Finally, the following are the border strip tableaux of shape $\lambda$ and type $\mu = (2,2,1,1)$.
 $$\begin{array}{ccc} 1 & 1 & 3 \\ 2 & 2 & \\ 4 & & \end{array} \qquad \begin{array}{ccc} 1 & 1 & 4 \\ 2 & 2 & \\ 3 & & \end{array} \qquad \begin{array}{ccc} 1 & 1 & 3 \\ 2 & 4 & \\ 2 & & \end{array} \qquad \begin{array}{ccc} 1 & 1 & 4 \\ 2 & 3 & \\ 2 & & \end{array}$$
 $$\begin{array}{ccc} 1 & 2 & 2 \\ 1 & 3 & \\ 4 & & \end{array} \qquad \begin{array}{ccc} 1 & 2 & 2 \\ 1 & 4 & \\ 3 & & \end{array} \qquad \begin{array}{ccc} 1 & 2 & 3 \\ 1 & 2 & \\ 4 & & \end{array} \qquad \begin{array}{ccc} 1 & 2 & 4 \\ 1 & 2 & \\ 3 & & \end{array}$$
 The first two and the last two tableaux have even height, and the rest have height $1$. This means that $\chi_\lambda(\pi) = 4 - 4 = 0$.
\end{exm}

Beautiful quantizations of the symmetric group and other Coxeter groups are the Hecke algebras, whose characters exhibit an even richer combinatorial structure. The formulas generalizing \eqref{intro2} and (implicitly) \eqref{intro1} were given by Ram \cite{RamFrob} (see also \cite{RamRemmel}).

\medskip

In this paper, we use bijective methods to rederive these combinatorial interpretations for type A, and to find analogous results for other types.

\medskip

The symmetric group $\f S_n$ is generated by transpositions $s_i = (i,i+1)$, $1 \leq i \leq n-1$, which satisfy the relations

\begin{alignat*}{2}
s_i^2 &= 1 &\qquad &\text{for } i=1,\ldots,n-1, \\
s_i s_j s_i &= s_j s_i s_j &\qquad &\text{if } |i-j|=1,\\
s_i s_j &= s_j s_i &\qquad &\text{if } |i-j| \geq 2.
\end{alignat*}

More generally, we can define a Coxeter group as follows. A symmetric matrix $M$ with rows and entries indexed by a finite set $S$ and with entries in $\set{1,2,\ldots,\infty}$ is called Coxeter if the diagonal elements are $1$ and the off-diagonal elements are strictly greater than $1$. The Coxeter group $(W,S)$ is the group generated by $S$ with relations
$$(ss')^{m(s,s')} = 1 \qquad \text{for all } s,s'.$$

\medskip

An expression $w = s_1 s_2\cdots s_k$, $s_k \in S$, is \emph{reduced} if it is the shortest such expression for $w$, and $k$ is called the length $\ell(w)$ of $w$. All reduced expressions contain the same generators, see \cite[Corollary 1.4.8 (ii)]{bb}. Denote by $\rw(w)$ the set of the generators contained in a reduced expression for $w$.

\medskip

Each Coxeter group has a corresponding Coxeter graph; we take one vertex for each generator, draw an edge between $s$ and $s'$ if $m(s,s') \geq 3$, and write $m(s,s')$ above the edge if $m(s,s') \geq 4$. We call a Coxeter group irreducible if its graph is connected. A Coxeter group is a product of irreducible Coxeter groups, and all possible finite irreducible Coxeter groups are well known: apart from a finite number of exceptional cases (which are not interesting for our purposes), there are four infinite families. They are shown in the following figure (the first three graphs have $n$ vertices).

\begin{figure}[ht]
  \begin{center}
   \input{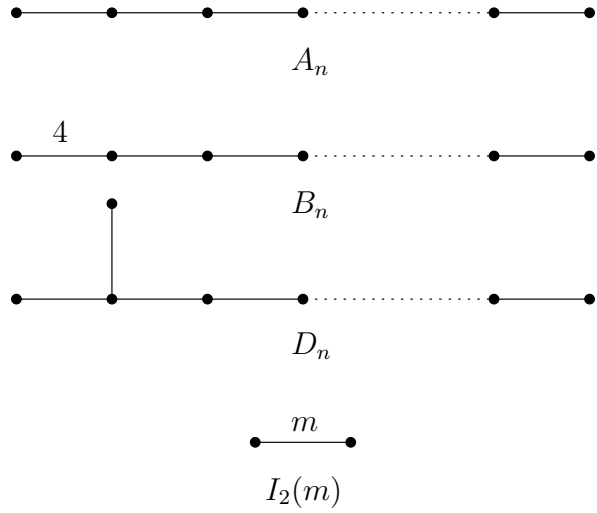}
  \end{center}
  \caption{Infinite families of irreducible Coxeter groups.}
  \label{fig1}
 \end{figure}

See \cite[\S A1]{bb}.

\medskip

For a Coxeter group $(W,S)$, we can define the corresponding Hecke algebra $H_{W,S}$ as follows. It is the $\mathbb{C}[q^{1/2}, q^{-1/2}]$-algebra generated by the set of {\em natural generators}, elements $\set{ T_{s} \colon s \in S}$ or, equivalently, by {\em modified natural generators} $\set{\wT_{s} \colon s \in S}$ with $\widetilde{T}_{s} = q^{-1/2} T_{s}$, subject to the relations
\begin{alignat*}{3}
T_{s}^2 &= (q-1) T_{s} + q   
&\qquad \wT_{s}^2 &= (q^{1/2}-q^{-1/2}) \wT_{s} + 1,    
&\qquad &\text{for } s \in S, \\
\underbrace{T_s T_{s'} T_s \cdots}_{m(s,s')} & = \underbrace{T_{s'} T_s T_{s'} \cdots}_{m(s,s')}
&\qquad \underbrace{\wT_s \wT_{s'} \wT_s \cdots}_{m(s,s')} & = \underbrace{\wT_{s'} \wT_s \wT_{s'} \cdots}_{m(s,s')}
&\qquad &\text{for }  s,s' \in S, s \neq s'.
\end{alignat*}

If $s_{i_1} \cdots s_{i_\ell}$ is a reduced expression for $w$ of length $\ell(w) = \ell$, we define
$$T_w = T_{s_{i_1}} \cdots T_{s_{i_\ell}}, \qquad \widetilde T_w = q^{-\ell/2}T_w = \widetilde T_{s_{i_1}} \cdots \widetilde T_{s_{i_\ell}}.$$

This is well defined (say, by Matsumoto's theorem, see \cite[Theorem 1.2.2]{gp}).

\medskip

For a Coxeter group $(W,S)$ and a subset $J$ of $S$, call the subgroup $W_J$ generated by $J$ a \emph{parabolic} subgroup (also a \emph{Young} subgroup when $W$ is the symmetric group). The subalgebra generated by $\set{T_s \colon s \in J}$ is called a parabolic (or Young) subalgebra.

\medskip

Let $\p A$ be an algebra over $\C$. A left $\p A$-module $\p V$ is called a \emph{representation} of $\p A$. We can also think of an algebra representation as an algebra homomorphism $\varphi = \varphi_{\p V} \colon \p A \to \End(\p V)$, where $\End(\p V)$ is the algebra of endomorphisms of $\p V$.

\begin{exm}
 Let us prove that the linear map $\eta \colon H_{(W,S)} \to \C$ defined by $\eta(\wT_w) = q^{\ell(w)/2}$ is a representation, i.e.\hspace{-0.07cm} that we have $\eta(\wT_w \wT_v) = q^{(\ell(w) + \ell(v))/2}$ for all $w,v$. This is obviously true if $v = e$, assume that it holds for all $w,v$ with $\ell(v) = k-1$, and assume $\ell(v) = k$. We have $v = s v'$ for some $s \in S$, $\ell(v') = k-1$. If $\ell (w s) = \ell(w) + 1$, then
 $$\eta(\wT_w \wT_v) = \eta(\wT_w \wT_{s} \wT_{v'}) = \eta(\wT_{ws} \wT_{v'}) = q^{(\ell(w s) + \ell(v'))/2} = q^{(\ell(w) + \ell(v))/2},$$
 and if $\ell (w s) = \ell(w) - 1$, then
 $$\eta(\wT_w \wT_v) = \eta(\wT_w \wT_{s} \wT_{v'}) = \eta((\wT_{ws}+(q^{1/2}-q^{-1/2})\wT_w) \wT_{v'}) = $$
 $$= \eta(\wT_{ws} \wT_{v'}) + (q^{1/2}-q^{-1/2}) \eta(\wT_w \wT_{v'}) = q^{(\ell(ws) + \ell(v'))/2} + (q^{1/2}-q^{-1/2}) q^{(\ell(w) + \ell(v'))/2} = $$
 $$= q^{(\ell(w) + \ell(v))/2 - 1} + (q^{1/2}-q^{-1/2}) q^{(\ell(w) + \ell(v) - 1)/2} = q^{(\ell(w) + \ell(v))/2}.$$
 This representation is called \emph{trivial}. We can similarly prove that $\epsilon \colon H_m(q) \to \C$, defined by $\epsilon(\wT_w) = (-q^{-1/2})^{\ell(w)}$, is a representation, we call it the \emph{sign representation}.
\end{exm}

For an $\p A$-module $\p V$ which is finitely generated and free over $\C$, the \emph{character} of $\p V$ is the linear map $\chi_{\p V} \colon \p A \longrightarrow \C$, $a \mapsto \tr(\varphi_{\p V}(a))$.

\medskip

First note the following. For a representation $\varphi$ and the corresponding character $\varphi$, we have
\begin{equation} \label{intro6}
 \chi(ab) = \tr(\varphi(ab))=\tr(\varphi(a)\varphi(b)) = \tr(\varphi(b)\varphi(a))=\tr(\varphi(ba))=\chi(ba).
\end{equation}
for every $a,b \in \p A$.

\medskip

The algebra $H_{(W,S)}$ is finitely generated and free (see \cite[\S 4.4]{gp}). Therefore it makes sense to talk about characters of its representations.

\medskip

Equation \eqref{intro6} implies the following relation for characters of Hecke algebras.

\begin{thm} \label{intro7}
 Take $w \in W$, $s \in S$, and a character $\chi$ of $H_{(W,S)}$. Then:
 \begin{itemize}
  \item if $\ell(sws) = \ell(w)$, then $\chi(\wT_{sws}) = \chi(\wT_w)$;
  \item if $\ell(sws) = \ell(w)+2$, then $\chi(\wT_{sws}) = \chi(\wT_w) + (q^{1/2}-q^{-1/2}) \chi(\wT_{sw}) = \chi(\wT_w) + (q^{1/2}-q^{-1/2}) \chi(\wT_{ws})$;
  \item if $\ell(sws) = \ell(w)-2$, then $\chi(\wT_{sws}) = \chi(\wT_w) - (q^{1/2}-q^{-1/2}) \chi(\wT_{sw}) = \chi(\wT_w) - (q^{1/2}-q^{-1/2}) \chi(\wT_{ws})$;
 \end{itemize}
\end{thm}
\begin{proof}
 Assume that $\ell(sw)=\ell(w) - 1$ and $\ell(sws) = \ell(w)$. Then
 $$\chi(\wT_s \wT_w \wT_s) = \chi((\wT_{sw} + (q^{1/2}-q^{-1/2})\wT_w) \wT_s) =\chi(\wT_{sws}) + (q^{1/2}-q^{-1/2}) \chi(\wT_w \wT_s)$$
 and, by \eqref{intro6},
 $$\chi(\wT_s \wT_w \wT_s) = \chi(\wT_w \wT_s \wT_s) = \chi(\wT_w (1 + (q^{1/2}-q^{-1/2}) \wT_s)) = \chi(\wT_w) + (q^{1/2}-q^{-1/2}) \chi(\wT_w \wT_s),$$
 so $\chi(\wT_{sws}) = \chi(\wT_w)$. If $\ell(sw)=\ell(w) + 1$ and $\ell(sws) = \ell(w)$, then
 $$\chi(\wT_s \wT_w \wT_s) = \chi(\wT_{sw} \wT_s) =\chi(\wT_{sws}) + (q^{1/2}-q^{-1/2}) \chi(\wT_{sw})$$
 and, by \eqref{intro6},
 $$\chi(\wT_s \wT_w \wT_s) = \chi(\wT_s \wT_s \wT_w) = \chi((1 + (q^{1/2}-q^{-1/2}) \wT_s) \wT_w) = \chi(\wT_w) + (q^{1/2}-q^{-1/2}) \chi(\wT_{sw}).$$
 This proves (a). Let us prove (b). If $\ell(sws) = \ell(w)+2$, then $\ell(sw) = \ell(ws) = \ell(w)+1$, and so
 $$\chi(\wT_s \wT_w \wT_s) = \chi(\wT_{sws}) = \chi(\wT_w \wT_s \wT_s) = $$
 $$= \chi(\wT_w (1 + (q^{1/2}-q^{-1/2}) \wT_s)) = \chi(\wT_w) + (q^{1/2}-q^{-1/2}) \chi(\wT_{ws}),$$
 and since $\chi(\wT_s \wT_w) = \chi(\wT_w \wT_s)$, we have $\chi(\wT_{sws}) = \chi(\wT_w) + (q^{1/2}-q^{-1/2}) \chi(\wT_{sw})$. Swapping the roles of $w$ and $sws$, we get (c) from (b).
\end{proof}

Choose $J \subseteq S$, and say that we are given a representation of $H_J$, i.e.\hspace{-0.07cm} a $H_J$-module $\p V$. The tensor product $H_{(W,S)} \otimes_{H_J} \p V$ is naturally an $H_{(W,S)}$-module with the action $h'(h \otimes v) = h'h \otimes v$. This is the \emph{induced representation}. See \cite[page 287]{gp}.

\medskip

Define
$$X_J = \set{x \in W \colon \ell(xs) > \ell(x) \mbox{ for all } s \in J}.$$
Every $x \in X_J$ is the unique element of minimal length in the coset $x W_J$ of $W$. Furthermore, for every $w \in W$, there exist unique $x \in X_J$ and $v \in W_J$ satisfying $w = xv$; we also have $\ell(xv) = \ell(vx^{-1}) = \ell(x) + \ell(v)$. See \cite[Proposition 2.1.1]{gp}.

\medskip

We will find interpretations of characters induced from characters of parabolic subalgebras of Hecke algebras of all infinite families of irreducible Coxeter groups, evaluated on certain elements of Hecke algebras. The proofs we give are completely bijective. We will also give a bijective proof of formulas for the irreducible characters of the Hecke algebra of type A. There, we need a quantization of the Murnaghan-Nakayama rule, Lemma \ref{a5} -- for which we find a proof with a natural bijection and a natural involution -- and the following fact.

\medskip

It is well known that the inverse Kostka numbers $K_{\mu,\lambda}^{-1}$, defined as the coefficient of the complete symmetric function $h_\mu$ in Schur function $s_\lambda$, describe the expansions of irreducible $\f S_n$ characters in terms of induced sign and trivial characters of $\f S_n$, i.e.
$$\chi_\lambda = \sum_{\mu} K_{\mu,\lambda}^{-1} \eta_\mu = \sum_{\mu} K_{\mu,\lambda'}^{-1} \epsilon_\mu.$$
Somewhat surprisingly, these numbers also describe the expansions of irreducible $H_n^A$ characters in terms of induced sign and trivial characters of $H_n^A$ in terms of irreducible $H_n^A$ characters; no ``quantum analogue'' of inverse Kostka numbers is needed for this purpose. See \cite[Sec.\,9.1.9]{gp}. In other words, we have
\begin{equation} \label{intro5}
 \chi_\lambda = \sum_{\mu} K_{\mu,\lambda}^{-1} \eta_\mu = \sum_{\mu} K_{\mu,\lambda'}^{-1} \epsilon_\mu,
\end{equation}
where $\chi_\lambda$ is the irreducible character of $H_n^A$ corresponding to $\lambda$, $\eta_\mu$ is the character induced from the trivial character on the Young subalgebra $H_\mu$, and $\epsilon_\mu$ is the character induced from the sign character on the Young subalgebra $H_\mu$.

\medskip

Throughout the paper, we will write
$$R = q^{1/2} - q^{-1/2}.$$

\section{Main results and examples} \label{theorems}

The basis for our computations is the following lemma.

\begin{lemma} \label{thm6}
 Take a Hecke algebra $H_{(W,S)}$, subset $J \subseteq S$, character $\chi_0$ on the subalgebra $H_J$, and an element $w \in W$. Furthermore, denote by $\chi$ the character on $H_{(W,S)}$ induced from $\chi_0$. Then
 $$\chi(T_w) = \sum_{x \in X_J} \sum_{u \in W_J} \chi_0(T_u) [T_{xu}] T_w T_x.$$
\end{lemma}
\begin{proof}
 Assume that $\chi_0$ is the character corresponding to the $H_J$-module $\p V$. Assume that $\set{v_1,\ldots,v_k}$ is a basis of $\p V$. Write $w = x u$ for $x\in X_J$ and $u \in S_J$. Then $T_w \otimes v \in H_{(W,S)} \otimes_{H_J} \p V$ can be expressed as
 $$T_x T_u \otimes v = T_x \otimes T_u v = T_x \otimes \sum_i c_i v_i = \sum_i c_i (T_x \otimes v_i).$$
 Therefore, the basis of $H_{(W,S)} \otimes_{H_J} \p V$ is $\set{T_x \otimes v_i \colon x\in X_J,i=1,\ldots,k}$. For a basis element $T_x \otimes v_i$, we have
 $$T_w \left( T_x \otimes v_i \right) = T_w T_x \otimes v_i = \sum_{z \in W} c_{w,x}^z T_z \otimes v_i = \sum_{y \in X_J} \sum_{u \in S_J} c_{w,x}^{yu} T_y T_u \otimes v_i = \sum_{y \in X_J} \sum_{u \in S_J} c_{w,x}^{yu} T_y  \otimes T_u v_i,$$
 where $c_{w,x}^z = [T_z] T_w T_x$. Therefore
 $$\sum_{i=1}^k [T_x \otimes v_i] T_w \left( T_x \otimes v_i \right) = \sum_{u \in S_J} c_{w,x}^{xu} \chi_0(T_u),$$
 and
 $$\chi(T_w) = \sum_{x \in X_J} \sum_{i=1}^k [T_x \otimes v_i] T_w \left( T_x \otimes v_i \right) = \sum_{x \in X_J} \sum_{u \in W_J} \chi_0(T_u) [T_{xu}] T_w T_x.$$
 This completes the proof.
\end{proof}

The same formula holds for the basis $\wT_w$.

\medskip

Throughout the following sections, we will use certain well-known facts about Coxeter groups of type A, B, D and I. The book \cite{bb} is a good reference for these results.

\medskip

The Coxeter group $A_{n-1}$ can be represented as the symmetric group $\f S_n$, consisting of permutations of the set $\set{1,\ldots,n}$. We write permutations either in one-line notation or in cycle notation. The group of generators is $S_n^A = \set{s_1,\ldots,s_{n-1}}$, where $s_i$ is the simple transposition $(i,i+1)$; we have $s_i^2=1$, $s_is_{i+1}s_i = s_{i+1}s_is_{i+1}$, $s_i s_j = s_j s_i$ for $|i-j|\geq 2$. See Figure \ref{fig1}. We multiply permutations from right to left, and denote the corresponding Hecke algebra by $H_n^A$.

\medskip

There is a natural correspondence between parabolic subgroups of $\f S_n$ and compositions of $n$. For $J \subseteq S_n^A$, draw $n$ dots, and place a bar between dots $i$ and $i+1$ if and only if $s_i \notin J$, and assign a composition to dots and bars in the standard fashion. For example, the composition corresponding to $J=\set{s_1,s_2,s_4,s_7}$ in $\f S_8$ is $3212$, and the generating set of the parabolic subgroup corresponding to $1421$ is $\set{s_2,s_3,s_4,s_6}$.

\medskip

We call an element of $\f S_n$ \emph{parabolic of type $K$} if it is a product of elements of $K \subseteq S_n^A$ (in any order). For a composition $\lambda = (\lambda_1,\ldots,\lambda_p) \vdash n$ corresponding to $J \subseteq S_n^A$, denote the (parabolic) subgroup $\f S_{\lambda_1} \times \cdots \times \f S_{\lambda_p}$ by $\f S_J$ or $\f S_\lambda$. The corresponding parabolic subalgebra will be denoted by $H_J$ or $H_\lambda$. Two parabolic elements of the same type are conjugate and of the same length.

\medskip

The following theorem was proved by Ram \cite{RamFrob}.

\begin{thm}[induction in $H_n^A$] \label{thm1}
 Say we are given subsets $J,K$ of $S_n^A$. Denote by $\lambda = (\lambda_1,\ldots,\lambda_p) \vdash n$ the composition corresponding to $J$, and by $\eta_\lambda$ (respectively, $\epsilon_\lambda$) the character of the representation of $H_n^A$ induced from the trivial (respectively, sign) representation of the parabolic subalgebra $H_J$. For a parabolic element $w$ of type $K$, we have
 $$\eta_\lambda(T_{w}) = \sum_a q^{e_K(a)} (q-1)^{d_K(a)}$$
 and
 $$\epsilon_\lambda(T_{w}) = \sum_a (-1)^{e_K(a)} (q-1)^{d_K(a)},$$
 where the sums are over all integer sequences $a = a_1 a_2 \cdots a_n$ satisfying
 \begin{enumerate}
  \item $1 \leq a_i \leq p$ for $i = 1,\ldots, n$,
  \item $\#\set{i \colon a_i = k} = \lambda_k$ for $k = 1,\ldots,p$,
  \item if $s_i \in K$, then $a_i \geq a_{i+1}$, for $i = 1,\ldots,n-1$,
 \end{enumerate}
 and where
 \begin{itemize}
  \item $d_K(a)$ is the number of elements in the set $\set{i \colon s_i \in K, a_i > a_{i+1}}$,
  \item $e_K(a)$ is the number of elements in the set $\set{i \colon s_i \in K, a_i = a_{i+1}}$.
 \end{itemize}
\end{thm}

Define a \emph{broken border strip} to be a (not necessarily connected) skew shape with no $2 \times 2$ square. Equivalently, a skew shape $\lambda/\mu$ is a broken border strip if and only if $\lambda_i \leq \mu_{i-1} + 1$ for $i \geq 2$. A broken border strip $T$ is a union of a finite number, $\strp(T)$, of border strips. Denote by $\hght(T)$ the sum of heights of these border strips and by $\wdth(T)$ the sum of their widths.

\medskip

A broken border strip tableau of shape $\lambda/\mu$ and type $\alpha=(\alpha_1,\ldots,\alpha_p)$ is an assignment of positive integers to the squares of $\lambda/\mu$ such that:
\begin{itemize}
 \item every row and column is weakly increasing,
 \item the integer $i$ appears $\alpha_i$ times, and
 \item the set of squares occupied by $i$ forms a broken border strip or is empty.
\end{itemize}
The weight of a broken border strip $T$ is
$$\weight T = (-1)^{\hght(T)} q^{\wdth(T)} (q-1)^{\strp(T) - 1},$$
and the weight $\weight \p T$ of a broken border strip tableau $\p T$ is the product of weights of its non-empty broken border strips.

\begin{thm} \label{thm5}
 For a parabolic element $w \in \f S_n$ of type $K$, denote by $\mu$ the corresponding composition of $n$. Denote the irreducible character of $H_n^A$ corresponding to $\lambda$ by $\chi_\lambda$. We have
 $$\chi_\lambda(T_{w}) = \sum_{\p T} \weight \p T,$$
 where the sum is over all broken border strip tableaux $\p T$ of shape $\lambda$ and type $\mu$.
\end{thm}

Let us state the result in type B. It is well known that we can represent the Coxeter group $B_n$ as $\f S_{-n}$, the group of \emph{signed permutations of $n$}. A signed permutation is a permutation $w$ of the set $\set{-n,\ldots,-1,0,1,\ldots,-n}$ satisfying $w(-i)=-w(i)$ for all $i$. It is uniquely determined by the sequence $w(1)w(2)\cdots w(n)$, which we call the \emph{one-line notation of $w$}. Since a signed permutation is a permutation, we can write it in cyclic notation. We write $(i_1,i_2,\ldots,i_k)$ instead of $(i_1,i_2,\ldots,i_k)(-i_1,-i_2,\ldots,-i_k)$ (and call such a cycle a \emph{positive} cycle) and $(i_1,i_2,\ldots,i_k)^-$ instead of $(i_1,i_2,\ldots,i_k,-i_1,-i_2,\ldots,-i_k)$ (and call such a cycle a \emph{negative} cycle). The group $\f S_{-n}$ is generated by $t = (1)^-$ and $s_i=(i,i+1)$, $i=1,\ldots,n-1$, which satisfy the relations $t^2=s_i^2=e$, $ts_1ts_1=s_1ts_1t$, $s_i s_{i+1}s_i = s_{i+1}s_i s_{i+1}$, $ts_i = s_i t$ for $i \geq 2$ and $s_i s_j = s_j s_i$ for $|i-j| \geq 2$. See Figure \ref{fig1}. Denote the set $\set{t,s_1,\ldots,s_{n-1}}$ by $S_n^B$, and the corresponding Hecke algebra by $H_n^B$.

\medskip

A parabolic subgroup of $\f S_{-n}$ is naturally isomorphic to either $\f S_{\lambda_1} \times \f S_{\lambda_2} \times \cdots \times \f S_{\lambda_p}$ or $\f S_{-\lambda_1} \times \f S_{\lambda_2} \times \cdots \times \f S_{\lambda_p}$ for a composition $\lambda \vdash n$, depending on whether $t$ is among the generators of the subgroup. Call a sequence of non-zero integers $(\lambda_1,\ldots,\lambda_p)$ a \emph{signed composition of $n$}, $\lambda \vdash n$, if $|\lambda_1| + \ldots + |\lambda_p| = n$. If we disregard the order of the elements of a signed composition, we get a \emph{signed partition}. The cycle type of a signed permutation can be thought of as a signed partition. Again, we call an element of $\f S_{-n}$ \emph{parabolic of type $K$} if it is a product of elements of $K \subseteq S_n^B$ (in any order).

\begin{rmk}
 Throughout this paper, it will often be useful to write $\underline n$ instead of $-n$ for a positive integer $n$.
\end{rmk}

\begin{thm}[induction in $H_n^B$] \label{thm2}
 Say we are given subsets $J,K$ of $S_n^B$. Denote by $\lambda = (\lambda_1,\ldots,\lambda_p) \vdash n$ the signed composition corresponding to $J$, and by $\eta_\lambda$ (respectively, $\epsilon_\lambda$) the character of the representation of $H_n^B$ induced from the trivial (respectively, sign) representation of the parabolic subalgebra $H_J$. For a parabolic element $w$ of type $K$, we have
 $$\eta_\lambda(T_{w}) = \sum_a q^{e_K(a) + g_K(a)} (q-1)^{d_K(a) + f_K(a)}$$
 and
 $$\epsilon_\lambda(T_{w}) = \sum_a (-1)^{e_K(a) + g_K(a)} (q-1)^{d_K(a) + f_K(a)},$$
 where the sums are over all integer sequences $a = a_1 a_2 \cdots a_n$ satisfying
 \begin{enumerate}
  \item $1 \leq |a_i| \leq p$ for $i = 1,\ldots, n$,
  \item $\#\set{i \colon |a_i| = k} = |\lambda_k|$ for $k = 1,\ldots,p$,
  \item if $t \in J$, then $a_i \neq -1$ for all $i = 1,\ldots,n$,
  \item if $s_i \in K$, then $a_i \geq a_{i+1}$ for $i = 1,\ldots,n-1$,
  \item if $t \in K$, then either $a_1<0$, or $a_1=1$ and $t \in J$.
 \end{enumerate}
 and where
 \begin{itemize}
  \item $d_K(a)$ is the number of elements in the set $\set{i \colon s_i \in K, a_i > a_{i+1}}$,
  \item $e_K(a)$ is the number of elements in the set $\set{i \colon s_i \in K, a_i = a_{i+1}}$,
  \item $f_K(a)$ is $1$ if $t \in K$ and $a_1<0$, and $0$ otherwise.
  \item $g_K(a)$ is $1$ if $t \in K$ and $a_1>0$, and $0$ otherwise.
 \end{itemize}
\end{thm}

Let us turn to type D. We can represent the Coxeter group $D_n$ as $\f S^0_{-n}$, the subgroup of $\f S_{-n}$ consisting of all signed permutations $w$ for which $\#\set{i \colon i > 0, w(i) < 0}$ is even. The group $\f S^0_{-n}$ is generated by $t = (1 -2)$ and $s_i=(i,i+1)$, $i=1,\ldots,n-1$, which satisfy the relations $t^2=s_i^2=e$, $ts_2t=s_2ts_2$, $s_i s_{i+1}s_i = s_{i+1}s_i s_{i+1}$, $ts_i = s_i t$ for $i \neq 2$ and $s_i s_j = s_j s_i$ for $|i-j| \geq 2$. Write $S_n^D = \set{t,s_1,\ldots,s_{n-1}}$. See Figure \ref{fig1}. Denote the set $\set{t,s_1,\ldots,s_{n-1}}$ by $S_n^D$, and the corresponding Hecke algebra by $H_n^D$.

\medskip

The generators $s_1$ and $t$ play a symmetric role. Therefore we may assume without loss of generality that we do \emph{not} have $s_1 \notin J$, $t \in J$. In this case, a parabolic subgroup of $\f S^0_{-n}$ is naturally isomorphic to either $\f S_{\lambda_1} \times \f S_{\lambda_2} \times \cdots \times \f S_{\lambda_p}$, $\f S^0_{-\lambda_1} \times \f S_{\lambda_2} \times \cdots \times \f S_{\lambda_p}$ for a composition $\lambda \vdash n$, depending on whether $t$ is among the generators of the subgroup.

\medskip

We do not impose such a limitation on $K$ in the following.

\begin{thm}[induction in $H_n^D$] \label{thm3}
 Say we are given subsets $J,K$ of $S_n^D$. Denote by $\lambda = (\lambda_1,\ldots,\lambda_p) \vdash n$ the signed composition corresponding to $J$, and by $\eta_\lambda$ (respectively, $\epsilon_\lambda$) the character of the representation of $H_n^D$ induced from the trivial (respectively, sign) representation of the parabolic subalgebra $H_J$. For a parabolic element $w$ of type $K$, we have 
 $$\eta_\lambda(T_{w}) = \sum_a q^{e_K(a) + g_K(a)} (q-1)^{d_K(a) + f_K(a)}$$
 and
 $$\epsilon_\lambda(T_{w}) = \sum_a (-1)^{e_K(a) + g_K(a)} (q-1)^{d_K(a) + f_K(a)},$$
 where the sums are over all integer sequences $a = a_1 a_2 \cdots a_n$ satisfying
 \begin{enumerate}
  \item $1 \leq |a_i| \leq p$ for $i = 1,\ldots, n$,
  \item $\#\set{i \colon |a_i| = k} = |\lambda_k|$ for $k = 1,\ldots,p$,
  \item $\#\set{i \colon a_i < 0}$ is even,
  \item if $t \in J$, then $a_i \neq -1$ for at most one $i$, and for such $i$ we have $a_j\neq 1$ for $j < i$,
  \item if $s_i \in K$, then either $a_i \geq a_{i+1}$, or $t \in J$, $a_i = -1$, $a_{i+1}=1$,
  \item if $t \in K$, then we have the following implications:
  \begin{itemize}
   \item $|a_1|>|a_2| \Rightarrow a_1<0$,
   \item $|a_1|<|a_2| \Rightarrow a_2<0$,
   \item $|a_1|=|a_2| \Rightarrow a_2<0 \mbox{ or } a_1=a_2=1, t \in J$,
  \end{itemize} 
 \end{enumerate}
 and where
 \begin{itemize}
  \item $d_K(a)$ is the number of elements in the set $\set{i \colon s_i \in K, a_i > a_{i+1}}$,
  \item $e_K(a)$ is the number of elements in the set $\set{i \colon s_i \in K, a_i \leq a_{i+1}}$,
  \item $f_K(a)$ is $1$ if $t \in K$ and $a_1 =a_2 < 0$ or $|a_1|\neq |a_2|$, and $0$ otherwise,
  \item $g_K(a)$ is $1$ if $t \in K$ and $a_1 = -a_2$ or $a_1=a_2=1$, and $0$ otherwise.
 \end{itemize}
\end{thm}

\begin{rmk}
 A permutation is parabolic if and only if it has minimal length in its conjugacy class. Similarly, a signed permutation in $\f S_{-n}$ or $\f S_{-n}^0$ is parabolic if has minimal length in its conjugacy class, and if its conjugacy class has at most one negative cycle. See Section \ref{remarks} for a discussion of the calculation of values of characters evaluated at $T_w$ for other (signed) permutations which are not of minimal length in their conjugacy classes.
\end{rmk}

The Coxeter group $I_2(m)$ is generated by two elements, which we denote $1$ and $2$, and the relation $(12)^m = (21)^m$. See Figure \ref{fig1}. Denote the set $\set{1,2}$ by $S_m^I$, and the corresponding Hecke algebra by $H_m^I$. The group has $2m$ elements, one of length $0$, two of length $i$ for $1 \leq i \leq m-1$, and one of length $m$. There are only four parabolic subgroups, generated by $\emptyset$, $\set 1$, $\set 2$ and $\set{1,2}$, with the last one being uninteresting for induction.

\begin{thm}[induction in $H_m^I$] \label{thm4}
 Denote by $H_0$ the trivial subalgebra, and by $H_1$ (respectively $H_2$) the parabolic subalgebra generated by $T_1$ (respectively $T_2$) of $H_m^I$. Denote by $\eta_c$ (respectively, $\epsilon_c$), $c=0,1,2$, the character of the representation of $H_m^I$ induced from the trivial (respectively, sign) representation of $H_c$. We have $\eta_0(T_e) = \epsilon_0(T_e) = 2m$, $\eta_1(T_e) = \epsilon_1(T_e) = \eta_2(T_e) = \epsilon_2(T_e) = m$. Furthermore, if $w$ is an element of even length $\ell \geq 2$, then
 $$\eta_0(T_w) = \epsilon_0(T_w)=\sum_{j=0}^{\ell/2-1} \frac{\ell}{\ell-j} \binom{\ell-j}{j} q^j (q-1)^{\ell-2j},$$
 and if $w$ is an element of odd length $\ell$, then
 $$\eta_0(T_w) = \epsilon_0(T_w)=m q^{(\ell+1)/2} - m q^{(\ell-1)/2} + \sum_{j=0}^{(\ell-3)/2} \frac{\ell}{\ell-j} \binom{\ell-j}{j} q^j (q-1)^{\ell-2j}.$$
 Now choose $c \in \set{1,2}$. For an element $w$ of even length $\ell \geq 2$, we have
 $$\eta_c(T_w) = \sum_{j=1}^{\ell/2} \binom{\ell/2+j-1}{2j-1} q^{\ell/2+1-j}(q-1)^{2j-1} + \sum_{j=1}^{\ell/2-1} \binom{\ell/2+j-1}{2j} q^{\ell/2-j} (q-1)^{2j}$$
 and
 $$\epsilon_c(T_w) = - \sum_{j=1}^{\ell/2} \binom{\ell/2+j-1}{2j-1} q^{\ell/2-j}(q-1)^{2j-1} + \sum_{j=1}^{\ell/2-1} \binom{\ell/2+j-1}{2j} q^{\ell/2-j} (q-1)^{2j}.$$
 For an element $w$ of odd length $\ell$, we have
 $$\eta_c(T_w) =  \left\{\begin{array}{ccl} \frac{m+1}2 q^{(\ell+1)/2} - \frac{m-1}2 q^{(\ell-1)/2} & : & m \mbox{ odd} \\ \left(\frac{m}2 + 1\right) q^{(\ell+1)/2} - \left(\frac{m}2  - 1 \right) q^{(\ell-1)/2} & : & m \mbox{ even, } w \sim c \\ \frac{m}2 q^{(\ell+1)/2} - \frac{m}2 q^{(\ell-1)/2} & : & m \mbox{ even, } w \not\sim c\end{array} \right.$$
 $$+ \sum_{j=2}^{(\ell-1)/2} \!\! \binom{(\ell-1)/2+j-1}{2j-1} q^{(\ell+1)/2-j}(q-1)^{2j-1} + \!\!\sum_{j=1}^{(\ell-1)/2} \!\! \binom{(\ell-1)/2+j}{2j} q^{(\ell+1)/2-j}(q-1)^{2j}$$
 and
 $$\epsilon_c(T_w) =  \left\{\begin{array}{ccl} \frac{m-1}2 q^{(\ell+1)/2} - \frac{m+1}2 q^{(\ell-1)/2} & : & m \mbox{ odd} \\ \left(\frac{m}2 - 1\right) q^{(\ell+1)/2} - \left(\frac{m}2  + 1 \right) q^{(\ell-1)/2} & : & m \mbox{ even, } w \sim c \\ \frac{m}2 q^{(\ell+1)/2} - \frac{m}2 q^{(\ell-1)/2} & : & m \mbox{ even, } w \not\sim c\end{array} \right.$$
 $$+ \sum_{j=2}^{(\ell-1)/2} \!\! \binom{(\ell-1)/2+j-1}{2j-1} q^{(\ell+1)/2-j}(q-1)^{2j-1} - \!\!\sum_{j=1}^{(\ell-1)/2} \!\! \binom{(\ell-1)/2+j}{2j} q^{(\ell-1)/2-j}(q-1)^{2j},$$
 where $w \sim v$ means that $w$ and $v$ are conjugate in $I_2(m)$.
\end{thm}

We illustrate the theorems with examples.

\begin{exm}[using Theorem \ref{thm1}]
 For $n=6$, $J=\set{s_1,s_2,s_4}$ (so $\lambda=321$) and $K = \set{s_1,s_3}$ (so $w=214356$), we have the following sequences and weights corresponding to $\eta_\lambda$. We insert bars in between elements of the sequence $a$ in places determined by $K$.
 $$\begin{array}{cc|cc|cc}
 a & \mbox{weight} & a & \mbox{weight} & a & \mbox{weight} \\
 \hline
 11|21|2|3 & q^{1}(q-1)^{1} & 11|21|3|2 & q^{1}(q-1)^{1} & 11|31|2|2 & q^{1}(q-1)^{1}\\ 
 11|22|1|3 & q^{2}(q-1)^{0} & 11|22|3|1 & q^{2}(q-1)^{0} & 11|32|1|2 & q^{1}(q-1)^{1}\\
 11|31|2|1 & q^{1}(q-1)^{1} & 21|11|2|3 & q^{1}(q-1)^{1} & 21|11|3|2 & q^{1}(q-1)^{1}\\
 21|21|1|3 & q^{0}(q-1)^{2} & 21|21|3|1 & q^{0}(q-1)^{2} & 21|31|1|2 & q^{0}(q-1)^{2}\\
 21|31|2|1 & q^{0}(q-1)^{2} & 21|32|1|1 & q^{0}(q-1)^{2} & 31|11|2|2 & q^{1}(q-1)^{1}\\
 31|21|1|2 & q^{0}(q-1)^{2} & 31|21|2|1 & q^{0}(q-1)^{2} & 31|22|1|1 & q^{1}(q-1)^{1}\\
 22|11|1|3 & q^{2}(q-1)^{0} & 22|11|3|1 & q^{2}(q-1)^{0} & 22|31|1|1 & q^{1}(q-1)^{1}\\
 31|11|1|2 & q^{1}(q-1)^{1} & 32|11|2|1 & q^{1}(q-1)^{1} & 32|21|1|1 & q^{0}(q-1)^{2}
 \end{array}$$
 That means that
 $$\eta_{321}(T_{214356}) = 4 q^2 + 12 q(q-1) + 8 (q-1)^2 = 24 q^2 - 28 q + 8$$
 and
 $$\epsilon_{321}(T_{214356}) = 4 - 12(q-1) + 8(q-1)^2 = 8 q^2 - 28 q + 24.$$
\end{exm}

\begin{exm}[using Theorem \ref{thm5}]
 For $n=6$, $J=\set{s_1,s_2,s_4}$ (so $\lambda=321$) and $K = \set{s_1,s_3}$ (so $w=214356$ and $\mu = 2211$), the following are all broken border strip tableaux of shape $\lambda$ and type $\mu$:
 $$\begin{array}{ccc} 1 & 1 & 3 \\ 2 & 2 & \\ 4 & & \end{array} \qquad \begin{array}{ccc} 1 & 1 & 4 \\ 2 & 2 & \\ 3 & & \end{array} \qquad \begin{array}{ccc} 1 & 1 & 3 \\ 2 & 4 & \\ 2 & & \end{array} \qquad \begin{array}{ccc} 1 & 1 & 4 \\ 2 & 3 & \\ 2 & & \end{array}$$
 $$\begin{array}{ccc} 1 & 2 & 2 \\ 1 & 3 & \\ 4 & & \end{array} \qquad \begin{array}{ccc} 1 & 2 & 2 \\ 1 & 4 & \\ 3 & & \end{array} \qquad \begin{array}{ccc} 1 & 2 & 3 \\ 1 & 2 & \\ 4 & & \end{array} \qquad \begin{array}{ccc} 1 & 2 & 4 \\ 1 & 2 & \\ 3 & & \end{array}$$
 $$\begin{array}{ccc} 1 & 1 & 2 \\ 2 & 3 & \\ 4 & & \end{array} \qquad \begin{array}{ccc} 1 & 1 & 2 \\ 2 & 4 & \\ 3 & & \end{array} \qquad \begin{array}{ccc} 1 & 2 & 3 \\ 1 & 4 & \\ 2 & & \end{array} \qquad \begin{array}{ccc} 1 & 2 & 4 \\ 1 & 3 & \\ 2 & & \end{array}$$
 Note that the first $8$ tableaux are actually border strip tableaux. We have
 $$\chi^{321}(T_{214356}) = 2q^4-4q^2+2+2q^2(q^2-1)-2(q^2-1) = 4q^4-8q^2+4.$$
\end{exm}

\begin{exm}[using Theorem \ref{thm2}]
 For $n=5$, $J=\set{t,s_1,s_2,s_4}$ (so $\lambda=\underline 221$) and $K = \set{t,s_1,s_3}$ (so $w= 2\underline 1 4 3 5$), we have the following sequences and weights corresponding to $\eta_\lambda$. We insert bars in between elements of the sequence $a$ in places determined by $K$.
 $$\begin{array}{cc|cc|cc}
 a & \mbox{weight} & a & \mbox{weight} & a & \mbox{weight} \\
 \hline
 11|1 \un 2| \un 2 & q^{2}(q-1)^{1} & 1 1|1 \un 2|2 & q^{2}(q-1)^{1} & 1 1|2 1|\un 2 & q^{2}(q-1)^{1}\\ 
 11|2 1| \un 2 & q^{2}(q-1)^{1} & 11| \un 2 \un 2|1 & q^{3}(q-1)^{0} & 11|2 \un 2|1 & q^{2}(q-1)^{1}\\
 11|22|1 & q^{3}(q-1)^{0} & 1 \un 2|11| \un 2 & q^{2}(q-1)^{1} & 1 \un 2|11|2 & q^{2}(q-1)^{1}\\
 1 \un 2|1 \un 2|1 & q^{1}(q-1)^{2} & 1 \un 2|21|1 & q^{1}(q-1)^{2} &  \un 2\un 2|11|1 & q^{2}(q-1)^{3}  \end{array}$$
 Thus
 $$\eta_{\un 221}(T_{2 \underline 14 3 5}) = 2q^3 + 8q^2(q-1)+2q(q-1)^2 = 12 q^3 - 12 q^2 + 2 q$$
  and
 $$\epsilon_{\un 221}(T_{2 \underline 14 3 5}) = -2 + 8q^2(q-1)-2q(q-1)^2 = 6 q^3- 4 q^2- 2 q-2.$$
\end{exm}

\begin{exm}[using Theorem \ref{thm3}]
 For $n=5$, $J=\set{t,s_1,s_2,s_4}$ (so $\lambda=\underline 221$) and $K = \set{t,s_1,s_3}$ (so $w=\underline 1 \underline 2 4 3 5$), we have the following sequences and weights corresponding to $\eta_\lambda$. We insert bars in between elements of the sequence $a$ in places determined by $K$.
 $$\begin{array}{cc|cc|cc}
  a & \mbox{weight} & a & \mbox{weight} & a & \mbox{weight} \\
  \hline
  2 \un 2|\un 1 1|1 & q^{2}(q-1)^{1} & \un 1 1|1 \un 2|2 & q^{2}(q-1)^{1} & \un 1 1|2 1|\un 2 & q^{2}(q-1)^{1}\\ 
  \un 11|2\un 2|1 & q^{2}(q-1)^{1} & \un 1 \un 2|11|2 & q^{1}(q-1)^{2} & \un 1 \un 2|21|1 & q^{0}(q-1)^{3}\\
  11|1 \un 2|\un 2 & q^{2}(q-1)^{1} & 11|21|2 & q^{2}(q-1)^{1} & 11|\un 2 \un 2|1 & q^{3}(q-1)^{0}\\
  11|22|1 & q^{3}(q-1)^{0} & 1 \un 2|11|\un 2 & q^{1}(q-1)^{2} & 1\un 2|1 \un 2|1 & q^{0}(q-1)^{3}\\
  \un 2 \un 2|11|1 & q^{2}(q-1)^{1} & & & &
  \end{array}$$
 Therefore
 $$\eta_{\un 221}(T_{\underline 1 \underline 2 4 3 5}) = 2q^3 + 7q^2(q-1)+2q(q-1)^2+2(q-1)^3 = 13 q^3 - 17 q^2 + 8 q  - 2$$
 and
 $$\epsilon_{\un 221}(T_{\underline 1 \underline 2 4 3 5}) = -2 + 7(q-1)-2(q-1)^2+2(q-1)^3 = 2 q^3- 8 q^2 + 17 q  -13.$$
\end{exm}

\begin{exm}[using Theorem \ref{thm4}]
 Take $m=14$ and $w=121212121$. Then
 $$\eta_1(T_w) = 8q^5 - 6q^4 + 10 q^3(q-1)^3+6q^2(q-1)^5+q(q-1)^7+$$
 $$+10q^4(q-1)^2+15q^3(q-1)^4+7q^2(q-1)^6+q(q-1)^8 = q^9+3 q^5-2 q^4$$
 and
 $$\epsilon_1(T_w) = 6q^5 - 8q^4 + 10 q^3(q-1)^3+6q^2(q-1)^5+q(q-1)^7-$$
 $$-10q^3(q-1)^2-15q^2(q-1)^4-7q^1(q-1)^6-(q-1)^8 = 2 q^5-3 q^4-1.$$
\end{exm}

\section{Characters in type A} \label{a}

For a subset $J \subseteq S_n^A$ and the corresponding composition $\lambda$, $\# \f S_J = \prod \lambda_i!$, and the index of $\f S_J$ is $\binom{n}{\lambda_1,\ldots,\lambda_p}$.

\medskip

The length of a permutation $w=w_1 \cdots w_n$ is equal to the number of inversions, i.e.
$$\#\set{i < j \colon w_i>w_j}.$$
Permutations $w$ and $v$ are in the same left coset of $\f S_J$ if and only if for each $k = 1,\ldots,p$, the sets 
$$\set{w_{\lambda_1 + \ldots + \lambda_{k-1} + 1},w_{\lambda_1 + \ldots + \lambda_{k-1} + 2},\ldots,w_{\lambda_1 + \ldots + \lambda_{k}}}$$
and 
$$\set{v_{\lambda_1 + \ldots + \lambda_{k-1} + 1},v_{\lambda_1 + \ldots + \lambda_{k-1} + 2},\ldots,v_{\lambda_1 + \ldots + \lambda_{k}}}$$
are equal.

\medskip
 
Let us denote by $X_J$ the set of coset representatives of $\f S_J$ of minimal length. A permutation $x$ is in $X_J$ if and only if for each $k = 1,\ldots,p$, we have 
$$x_{\lambda_1 + \ldots + \lambda_{k-1} + 1} < x_{\lambda_1 + \ldots + \lambda_{k-1} + 2} < \ldots < x_{\lambda_1 + \ldots + \lambda_{k}}.$$

For a subset $J \subseteq S_n^A$ and the corresponding composition $\lambda = (\lambda_1,\ldots,\lambda_p) \vdash n$, there is a natural bijection $\Phi_J$ between left cosets of $\f S_J$ and integer sequences $a = a_1 a_2 \cdots a_n$ satisfying
\begin{enumerate}
 \item $1 \leq a_i \leq p$ for $i = 1,\ldots, n$,
 \item $\#\set{i \colon a_i = k} = \lambda_k$ for $k = 1,\ldots,p$.
\end{enumerate}
For a coset $v \f S_J$, take $a_i = k$ if and only if
$$\lambda_1+\ldots+\lambda_{k-1} + 1 \leq v^{-1}(i) \leq \lambda_1+\ldots+\lambda_{k}.$$
If $w = v u$ for $u \in \f S_J$, then $w^{-1}(i) = u^{-1} (v^{-1}(i))$; that means that the map is well defined. It is also easy to see that it is a bijection. By slight abuse of notation, we also use $\Phi_J$ as the corresponding bijection between $X_J$ and integer sequences with properties (1) and (2).

\begin{exm}
 Take $n=8$, $J=\set{s_1,s_2,s_4,s_6,s_7}$ (hence $\lambda=323$) and $v = 41573286$. Since
 $$1 \leq  v^{-1}(4),v^{-1}(1),v^{-1}(5) \leq 3 = \lambda_1,$$
 $$\lambda_1+1 = 4\leq v^{-1}(7),v^{-1}(3) \leq 5 = \lambda_1+\lambda_2, $$
 $$\lambda_1+\lambda_2+1 = 6 \leq v^{-1}(2),v^{-1}(8),v^{-1}(6) \leq 8,$$
 the corresponding coset $v \f S_J$ maps to $13211323$ with $\Phi_J$. On the other hand, the sequence $31213321$ is the image of the coset that contains $82437561$.
\end{exm}

We can describe the bijection as follows. Pick a coset representative $v$, write it in one-line notation, and draw bars in positions given by $\lambda$. In the example above, that would be $415|73|286$. The $i$-th element of $\Phi_J(v \f S_J)$ tells us in which of the ``slots'' $i$ is located. For example, $7$ is located in the second slot, so $a_7=2$. The inverse is similarly intuitive. Write the locations of $1$'s in $a=a_1\cdots a_n$ in increasing order, then the locations of $2$'s in increasing order, etc. The resulting permutation is the minimal length representative of the coset $\Phi_J^{-1}(a)$. For the example $31213321$ above, this minimal coset representative is $24837156$.

\medskip

The crucial lemma is the following.

\begin{lemma} \label{a1}
 Let $\sim_J$ denote the relation of being in the same coset of the subgroup $\f S_J$. For $s \in S_n^A$, and every $x,v \in \f S_n$, we have
 $$vsx \sim_J x \not\sim_J s x \Longrightarrow s \in \rw(v).$$
\end{lemma}
\begin{proof}
 Suppose that $s = s_i \notin \rw(v)$ and that $x \not\sim_J s x$. The latter means that 
 $$\lambda_1+\ldots+\lambda_{k-1} + 1 \leq x^{-1}(i) \leq \lambda_1+\ldots+\lambda_{k}$$
 and
 $$\lambda_1+\ldots+\lambda_{j-1} + 1 \leq x^{-1}(i+1) \leq \lambda_1+\ldots+\lambda_{j},$$
 where $k \neq j$ ($i$ and $i+1$ are in different ``slots''). Also, we have
 $$\lambda_1+\ldots+\lambda_{j-1} + 1 \leq (sx)^{-1}(i) \leq \lambda_1+\ldots+\lambda_{j}$$
 and
 $$\lambda_1+\ldots+\lambda_{k-1} + 1 \leq (sx)^{-1}(i+1) \leq \lambda_1+\ldots+\lambda_{k},$$
 On the other hand, $s_i \notin \rw(v)$ implies that $v$ permutes the elements $1,\ldots,i$ among themselves, and $i+1,\ldots,n$ among themselves. In other words, the sets 
 $$\set{h \colon h \leq i, \lambda_1+\ldots+\lambda_{j-1} + 1 \leq (sx)^{-1}(h) \leq \lambda_1+\ldots+\lambda_{j}}$$
 and
 $$\set{h \colon h \leq i, \lambda_1+\ldots+\lambda_{j-1} + 1 \leq (vsx)^{-1}(h) \leq \lambda_1+\ldots+\lambda_{j}}$$
 have the same number of elements. But 
 $$\set{h \colon h \leq i, \lambda_1+\ldots+\lambda_{j-1} + 1 \leq (sx)^{-1}(h) \leq \lambda_1+\ldots+\lambda_{j}}$$
 has one element (namely $i$) more than
 $$\set{h \colon h \leq i, \lambda_1+\ldots+\lambda_{j-1} + 1 \leq x^{-1}(h) \leq \lambda_1+\ldots+\lambda_{j}},$$
 so $vsx \not\sim_J x$. 
\end{proof}

We can use this as follows. Denote by $\wP_{x,J}$ the projection
$$\wP_{x,J}(\wT_v) = \left\{ \begin{array}{ccl} \wT_v & : & v \sim_J x \\ 0 & : & \mbox{otherwise} \end{array} \right..$$ 

\begin{cor} \label{a2}
 Write $w = w's$, where $s \notin \rw(w')$. Then  
 \begin{equation} \label{a6}
  \wP_{x,J} \left(\wT_{w's} \cdot \wT_x\right) = \wP_{x,J}\left(\wT_{w'}\cdot \wP_{x,J}\left(\wT_s \cdot \wT_x\right)\right).
 \end{equation}
\end{cor}
\begin{proof}
 We have $\wT_{w's} \wT_x = \wT_{w'} \left(\wT_s \wT_x\right)$. Recall that if $\ell(sx)>\ell(x)$, then $\wT_s \wT_x = \wT_{sx}$, and if $\ell(sx)<\ell(x)$, then $\wT_s \wT_x = \wT_{sx} + R \wT_x$, where $R = q^{1/2}-q^{-1/2}$. If $sx \sim_J x$, then $\wP_{x,J}\left(\wT_s \cdot \wT_x\right) = \wT_s \wT_x$, and the equality follows. On the other hand, if $sx \not\sim_J x$, then, by the lemma, $vsx\not\sim_J x$ for every subword $v$ of $w'$. But $\wT_{w'}\wT_{sx}$ is a linear combination of $\wT_{vsx}$ for $v$ a subword of $w'$, and therefore $\wP_{x,J}\left(\wT_{w'}\wT_{sx}\right) = 0$. If $\ell(sx)>\ell(x)$, then both sides of \eqref{a6} are equal to $0$, and if $\ell(sx)<\ell(x)$, they are both equal to 
 $$R \wP_{x,J} \left(\wT_{w'} \cdot \wT_x\right).$$
 This completes the proof.
\end{proof}

An immediate consequence of the corollary is the following formula, which holds if $s \notin \rw(w'),\rw(w'')$ and $\rw(w') \cap \rw(w'') = \emptyset$:

\begin{equation} \label{a3}
 \wP_{x,J} \left(\wT_{w'sw''} \cdot \wT_x\right) = \wP_{x,J}\left(\wT_{w'}\cdot \wP_{x,J}\left(\wT_s \cdot \wP_{x,J}\left(\wT_{w''} \cdot \wT_x\right)\right)\right).
\end{equation}

Indeed, this is exactly Corollary \ref{a2} when $\ell(w'')=0$, assume that it holds when $\ell(w'')=\ell$. Then
$$\wP_{x,J} \left(\wT_{w'sw'''s'} \cdot \wT_x\right) = \wP_{x,J} \left(\wT_{w'sw'''} \cdot \wP_{x,J}\left(\wT_{s'}\wT_x\right)\right)= $$
$$= \wP_{x,J}\left(\wT_{w'}\cdot \wP_{x,J}\left(\wT_s \cdot \wP_{x,J}\left(\wT_{w'''} \cdot \wP_{x,J}\left(\wT_{s'} \wT_x\right)\right)\right)\right) = $$
$$= \wP_{x,J}\left(\wT_{w'}\cdot \wP_{x,J}\left(\wT_s \cdot \wP_{x,J}\left(\wT_{w'''s'} \cdot \wT_x\right)\right)\right),$$
which proves it for $\ell(w'')=\ell+1$.

\medskip

Now assume that we are given a parabolic element $w$ of type $K$. For a transversal element $x \in X_J$, write $a = \Phi_J(x)$. We have one of the following two cases.

\medskip

Suppose first that there is an $i$ such that $s_i$ appears in a reduced word for $w$, $w = w' s_i w''$, and such that $k = a_i < a_{i+1}$. By definition, that means that
$$x^{-1}(i) \leq \lambda_1 + \ldots + \lambda_k < x^{-1}(i+1).$$ 
Furthermore, every $x' \sim_J x$ also satisfies 
$$(x')^{-1}(i) \leq \lambda_1 + \ldots + \lambda_k < (x')^{-1}(i+1).$$ 
In particular, $\ell(sx') > \ell(x')$ (which means $\wT_s \wT_{x'} = \wT_{sx'}$) and $sx' \not\sim_J x' \sim_J x$ (which means $\wP_{x,J}\left(\wT_{sx'}\right) = 0$). In other words, $\wP_{x,J}\left(\wT_s \cdot \wP_{x,J}\left(\wT_{w''} \cdot \wT_x\right)\right) = 0$, and equation \eqref{a3} implies that
$$\wP_{x,J} \left(\wT_{w} \wT_x\right) = 0.$$ 

\medskip

Now suppose that for every $i$ for which $s_i$ appears in a reduced word for $w$, $w = w' s_i w''$, we have $a_i \geq a_{i+1}$. We prove by induction on $\ell(w) = \# K$ that in this case,
$$\wP_{x,J} \left(\wT_{w} \wT_x\right) = R^{d_K(a)} \wT_{w_{x,J} x},$$
where $d_K(a)$ denotes the number of elements in the set $\set{i \colon s_i \in K, a_i > a_{i+1}}$, $w_{x,J}$ is the subword of $w$ consisting of $s_i \in J$ with $a_i = a_{i+1}$, and $\ell(w_{x,J} x) = \ell(w_{x,J}) + \ell(x)$. 

\medskip

The claim is obvious for $w = e$. Suppose that it holds for $w'$, and suppose that $w = sw'$ with $\ell(w) = \ell(w') + 1$. We have to evaluate
$$\wP_{x,J} \left(\wT_{w} \wT_x\right) = \wP_{x,J} \left(\wT_{s} \wT_{w'} \wT_x\right) = \wP_{x,J} \left( \wT_s \wP_{x,J}\left(\wT_{w'} \wT_x \right)\right) = R^{d_{K'}(a)} \wP_{x,J}\left( \wT_s \wT_{w_{x,J}' x}\right),$$
where $K' = K \setminus \set s$, $w_{x,J}'$ is the subword of $w'$ consisting of $s_i$ with $a_i = a_{i+1}$. If $s = s_i$ and $a_i > a_{i+1}$, then $v^{-1}(i+1) < v^{-1}(i)$, $s_i v \not\sim_J v$ and $\ell(s_i v) < \ell(v)$ for every $v \sim_J x$. That implies $\ell(s_i w_{x,J}' x) < \ell(w_{x,J}'x)$ and $\wP_{x,J}\left( \wT_s \wT_{w_{x,J}' x}\right) = R \wT_{w_{x,J}' x}$. This means that
$$\wP_{x,J} \left(\wT_{w} \wT_x\right) = R^{d_{K'}(a)+1} \wT_{w_{x,J}' x} = R^{d_K(a)} \wT_{w_{x,J} x}.$$
On the other hand, assume $s = s_i$ and $a_i = a_{i+1} = k$. That means 
$\lambda_1+\ldots+\lambda_{k-1} + 1 \leq x^{-1}(i), x^{-1}(i+1) \leq \lambda_1+\ldots+\lambda_{k}$. Also, because $x \in X_J$, we have $x^{-1}(i+1) = x^{-1}(i) + 1$. Since $w_{x,J}'$ does not contain $s_i$, it permutes the elements $1,\ldots,i$ among themselves, and the elements $i+1,\ldots,n$ among themselves. Again, since $x \in X_J$ and $w_{x,J}' x \sim_J x$, we have $(w_{x,J}' x)^{-1}(i) < (w_{x,J}' x)^{-1}(i+1)$. Therefore $\ell(s_i w_{x,J}' x) > \ell(w_{x,J}' x)$, $\wT_{s_i} \wT_{w_{x,J}' x} = \wT_{s_i w_{x,J}' x}$ and
$$\wP_{x,J} \left(\wT_{w} \wT_x\right) = R^{d_{K'}(a)} \wT_{s_iw_{x,J}' x} = R^{d_K(a)} \wT_{w_{x,J} x}.$$
Also, $\ell(s_i w_{x,J}' x) = \ell(w_{x,J} x) = \ell(w_{x,J}' x) + 1 = \ell(w_{x,J}') + \ell(x) + 1 = \ell(w_{x,J}) + \ell(x)$.

\medskip

Theorem \ref{thm1} follows immediately. Indeed, by Lemma \ref{thm6}, we have
$$\eta_\lambda(\wT_w) = \sum_{x \in X_J} \sum_{u \in \f S_J} q^{\ell(u)/2} [\wT_{xu}] \wT_w \wT_x = $$
$$=\sum_{x \in X_J} \sum_{u \in \f S_J} q^{\ell(u)/2} [\wT_{xu}] \wP_{x,J}(\wT_w \wT_x) = \sum_{x \in X_J} q^{\ell(w_{x,J})/2} R^{d_K(a)} = \sum_a q^{e_K(a)/2} R^{d_K(a)},$$
where the sum is over all sequences $a$ satisfying (1), (2) and (3) from Theorem \ref{thm1}. But then
$$\eta_\lambda(T_w) = \sum_a q^{\ell(w)/2} q^{e_K(a)/2} (q^{1/2}-q^{-1/2})^{d_K(a)} = \sum_a q^{e_K(a)} (q-1)^{d_K(a)},$$
where we used the fact that $\ell(w) = e_K(a) + d_K(a)$. The proof for $\epsilon_\lambda$ is analogous.

\medskip

We can use our description of characters induced from Young subalgebras to prove combinatorial formulas for irreducible characters $\chi_\lambda$ of $H_n^A$. In order to do this, we need \eqref{intro5} and the following lemma. Note that this result was already proved in \cite{RamRemmel}, see equation (22) and the remark following it. Our proof, though essentially equivalent, is completely direct and elementary (in particular, it does not need Littlewood-Richardson rule, Pieri rule or $\lambda$-ring manipulations). See also \cite{RRW}.

\medskip

Recall that the ordinary Murnaghan-Nakayama rule states that for any partition $\mu$ and $r \in \mathbb{N}$, we have
\begin{equation} \label{a4}
s_\mu \cdot p_r = \sum_\lambda (-1)^{\hght(\lambda/\mu)} s_\lambda, 
\end{equation}
where the sum is over all partitions $\lambda \supseteq \mu$ for which $\lambda/\mu$ is a border strip of size $r$. See \cite[Theorem 7.17.1]{StanEC}.

\medskip

Define \emph{quantum power symmetric function} $\overline p_\mu(y_1,y_2,\ldots)$ for a composition $\mu$ by
$$\overline p_\mu = \overline p_{\mu_1} \cdots \overline p_{\mu_s},$$
where
$$\overline p_r = \sum_{J} q^{N_=(J)} (q-1)^{N_<(J)} y_{i_1}y_{i_2} \cdots y_{i_n};$$
here $J$ runs over multisets $(i_1,\ldots,i_r)$ with $1 \leq i_1 \leq \ldots \leq i_r \leq n$, and $N_=(J) = \# \set{j \colon i_j = i_{j+1}}$ and $N_<(J) = \# \set{j \colon i_j < i_{j+1}}$. For example,
$$\overline p_3 = q^2 m_3 + q(q-1) m_{21} + (q-1)^2 m_{111}.$$
If $q=1$, we get ordinary power symmetric functions $p_\mu$. Furthermore, define
$$\widetilde p_\mu = \widetilde p_{\mu_1} \cdots \widetilde p_{\mu_s},$$
where
$$\widetilde p_r = \sum_{J} (-1)^{N_=(J)} (q-1)^{N_<(J)} y_{i_1}y_{i_2} \cdots y_{i_n};$$
here $J$ runs over the same set of multisets. For example,
$$\overline p_3 = m_3 - (q-1) m_{21} + (q-1)^2 m_{111}.$$
If $q=1$, we get $\sigma_\mu p_\mu$, where $\sigma_\mu$ is $1$ if the number of even parts of $\mu$ is even, and $-1$ otherwise.
 
\medskip

The following generalizes \eqref{a4}.

\begin{lemma}[quantum Murnaghan-Nakayama rule] \label{a5}
 For any partition $\mu$ and $r \in \mathbb{N}$ we have
 \begin{equation} \label{a7}
  s_\mu \cdot \overline p_r = \sum_\lambda (-1)^{\hght(\lambda/\mu)} q^{\wdth(\lambda/\mu)} (q-1)^{\strp(\lambda/\mu) - 1} s_\lambda, 
 \end{equation}
 where the sum runs over all partitions $\lambda \supseteq \mu$ for which $\lambda/\mu$ is a broken border strip of size $r$.
\end{lemma}
\begin{proof}
 We have to find a quantum version of the proof of \cite[Theorem 7.17.1]{StanEC}. Fix $n$, let $\delta = (n-1,n-2,\ldots,0)$, and, for $\alpha \in \mathbb{N}^n$, write $a_\alpha = \det(y_i^{\alpha_j})_{i,j=1}^n$. The classical definition of Schur functions says that $a_{\lambda+\delta}/a_\delta = s_\lambda(y_1,\ldots,y_n)$. It is therefore enough to prove
 $$a_{\mu + \delta} \cdot \overline p_r = \sum_\lambda (-1)^{\hght(\lambda/\mu)} q^{\wdth(\lambda/\mu)} (q-1)^{\strp(\lambda/\mu) - 1} a_{\lambda + \delta},$$
 where the sum goes over all partitions $\lambda \supseteq \mu$ for which $\lambda/\mu$ is a broken border strip of size $r$ and $n$ is at least the number of parts of $\lambda$, and let $n \to \infty$ in order to prove the lemma. Throughout the proof, all functions depend on $y_1,\ldots,y_n$.\\
 It is easy to see that for a partition $\nu$ with $n$ parts, we have
 $$a_{\mu + \delta} \cdot m_\nu = \sum a_{\mu + \delta + \sigma(\nu)},$$
 where the sum runs over all permutations $\sigma$ of $n$ (here $\sigma(\nu)$ is the composition we get if we shuffle the entries of $\nu$ according to $\sigma$, i.e. $\sigma(\nu)_{\sigma(i)} = \nu_i$). For example, for $\mu = 31$, $n=4$ and $\nu = 2210$, then
 $$a_{6310} \cdot m_{221} = $$
 $$= a_{8520} + a_{8511} + a_{8430} + a_{8412} + a_{8331} + a_{8322} + a_{7332} + a_{7512} + a_{7530} + a_{6531} + a_{6522} + a_{6432}.$$
 Of course, for every composition $\alpha$, $a_\alpha$ is equal to $\pm a_\mu$ for some partition $\mu$, and if $\alpha$ has a repeated part, then $a_\alpha = 0$. For example,
 $$a_{6310} m_{221} = a_{8520} + a_{8430} - a_{8421} - a_{7521} + a_{7530} + a_{6531} + a_{6432}.$$
 Let us find the coefficient of $a_{\lambda + \delta}$ in $a_{\mu + \delta} \cdot \overline p_r$. Assume without loss of generality that $\lambda$ and $\mu$ have $n$ parts (some of which can be $0$). We divide the calculations into two parts.\\
 Assume first that $\lambda/\mu$ is a broken border strip tableau (of size $r$). As a running example, let us take $\lambda = 5431$, $\mu = 33$, $r = 7$ and $n = 4$. We want to find the coefficient of $a_{8641}$ in $a_{6510} \cdot \overline p_7$. Since $\lambda/\mu$ is a broken border strip tableau, we have $\lambda_i \leq \mu_{i-1} + 1$ for $i \geq 2$. In other words, $(\lambda + \delta)_i \leq (\mu + \delta)_{i-1}$. We want to find all partitions $\nu$ (with possible zeros at the end) of $r$ so that $\sigma(\nu) + \mu + \delta$ is a permutation of $\lambda + \delta$ for some $\sigma$. Equivalently, we want to find all compositions $\nu$ so that $\pi(\nu + \mu + \delta) = \lambda + \delta$ for some permutation $\pi \in \f S_n$. In our example, $\pi_1((2,1,3,1) + (6,5,1,0)) = \pi_2((0,3,3,1) + (6,5,1,0)) = \pi_3((2,1,0,4) + (6,5,1,0)) = \pi_4((0,3,0,4) + (6,5,1,0)) = (8,6,4,1)$ for $\pi_1 = 1234$, $\pi_2 = 2134$, $\pi_3 = 1243$ and $\pi_4 = 2143$. Note that the signs of these permutations are $1,-1,-1,1$, respectively.\\
 Since $(\nu + \mu + \delta)_{i-1} \geq (\mu + \delta)_{i-1} \geq (\lambda + \delta)_i$, we must have $\pi(i) \leq i+1$ for $i \leq n-1$. Furthermore, if $\lambda_i \leq \nu_{i-1}$, then $\pi(i) \leq i$. Denote by $\p I \subseteq \set{2, \ldots, n }$ the set of $i$ with $\lambda_i = \nu_{i-1} + 1$. In our example, $\p I = \set{2,4}$. Note that the elements in $\p I$ correspond to rows that contain cells of the broken border strip $\lambda/\mu$, but are not the first row of a border strip of $\lambda/\mu$. Furthermore, denote by $\p K$ the set of all $i$ with $\lambda_i = \nu_i$. The elements of $\p K$ correspond to empty rows of $\lambda/\mu$.\\
 To a composition $\nu$ with $\pi(\nu + \mu + \delta) = \lambda + \delta$ for some $\pi$, assign $\p I_\nu \subseteq \p I$ by $\p I_\nu = \set{ i \in \p I \colon \pi(i) = i + 1}$. In our example, we have $\p I_{2131} = \emptyset$, $\p I_{0331} = \set{ 2}$, $\p I_{2104} = \set{ 4}$, $\p I_{0304} = \set{ 2,4}$. It is easy to see that this assignment is a bijection between compositions $\nu$ for which $\pi(\nu + \mu + \delta) = \lambda + \delta$ for some permutation $\pi \in \f S_n$, and subsets of $\p I$. It remains to figure out the appropriate sign and weight, and to sum over all subsets of $\p I$. In the running example, the weights of $m_{3211},m_{3310},m_{4210},m_{4300}$ in $\overline p_7$ are $q^3(q-1)^3,q^4(q-1)^2,q^4(q-1)^2,q^5(q-1)$, respectively, so the coefficient of $a_{5431}$ in $a_{33} \cdot \overline p_7$ is
 $$q^3(q-1)^3 - q^4(q-1)^2 - q^4(q-1)^2 + q^5(q-1) = q^3(q-1)\left( (q-1)^2 - 2q(q-1) + q^2\right) = $$
 $$= q^3(q-1) (q - (q-1))^2 = q^3(q-1).$$
 Note that $\lambda/\mu$ is composed of two border strips of widths $2$ and $1$ and heights $1$ and $1$, so the result matches with \eqref{a7}.\\
 For $\p J \subseteq \p I$, the corresponding $\pi$ satisfies $\pi(i) = i+1$ for $i \in \p J$, and the remaining elements appear in increasing order in $\pi$. In other words, the disjoint cycle decomposition of $\pi$ is of the form $(1, 2, \ldots, i_1)(i_1+1,i_1+2,\ldots,i_2)\cdots$, where $i_0=1,i_1,i_2,\ldots$ are precisely the elements of $\set{ 1,\ldots,n} \setminus \p J$. Since cycles of odd length are even permutations and cycles of even length are odd permutations, that means that the sign of $\pi$ is $(-1)^{|\p J|}$.\\
 Recall that the weight of $m_\nu$ in $\overline p_r$ is $q^{r - s} (q-1)^{s - 1}$, where $s$ is the number of different non-zero parts of $\nu$. That means that for a subset $\p I_\nu$ of $\p I$, $m_\nu$ appears with weight $q^{r - n + |\p I_\nu| + |\p K|} (q-1)^{n - |\p I_\nu| - |\p K| - 1}$. In turn, this implies that the coefficient of $a_{\lambda + \delta}$ in $a_{\mu + \delta} \cdot \overline p_r$ is
 $$\sum_{\p J \subseteq \p I} (-1)^{|\p J|} q^{r - n + |\p J| + |\p K|} (q-1)^{n - |\p J| - |\p K| - 1} = $$
 $$=(-1)^{|\p I|} q^{r-n + |\p K|} (q-1)^{n-|\p I| - |\p K| - 1} \sum_{k = 0}^{|\p I|} \binom {|\p I|} k (-1)^{|\p I| - k} q^k (q-1)^{|\p I| - k} =$$
 $$=  (-1)^{|\p I|} q^{r-n + |\p K|} (q-1)^{n-|\p I| - |\p K| - 1} \left( q - (q-1) \right)^{|\p I|} = (-1)^{\hght(\lambda/\mu)} q^{\wdth(\lambda/\mu)} (q-1)^{\strp(\lambda/\mu) - 1}.$$
 The second part of the proof deals with the case when $\lambda/\mu$ is \emph{not} a broken border strip. Let us start with an example. Choose $\lambda=(6,5,4,3,2)$, $\mu = (4,2,2,2)$, $r = 10$ and $n = 5$. Then $\lambda + \delta = (10,8,6,4,2)$ and $\mu + \delta = (8,5,4,3,0)$, so we get the following table of compositions $\nu$ and permutations $\pi$ for which $\pi(\nu + \mu + \delta) = \lambda + \delta$ ($\weight(\nu)$ denotes the coefficient of $m_\nu$ in $\overline p_r$):
 $$\begin{array}{c|c|c|c}
 \nu & \pi & \sign \pi & \weight(\nu) \\
 \hline
 23212 & 12345 & +1 & q^{5}(q-1)^{4} \\ 
 23032 & 12435 & -1 & q^{6}(q-1)^{3} \\
 21412 & 13245 & -1 & q^{5}(q-1)^{4} \\
 21052 & 13425 & +1 & q^{6}(q-1)^{3} \\
 05212 & 21345 & -1 & q^{6}(q-1)^{3} \\
 05032 & 21435 & +1 & q^{7}(q-1)^{2} \\
 01612 & 23145 & +1 & q^{6}(q-1)^{3} \\
 01072 & 23415 & -1 & q^{7}(q-1)^{2}
 \end{array}$$
 The involution
 $$12345 \stackrel \varphi \longleftrightarrow 13245, \quad 12435 \stackrel \varphi \longleftrightarrow 13425, \quad 21345 \stackrel \varphi \longleftrightarrow 23145, \quad 21435 \stackrel \varphi \longleftrightarrow 23415$$
 reverses signs and preserves weights, so the total coefficient of $a_{65432}$ in $a_{4222} \cdot \overline p_{10}$ is $0$.\\
 In general, a sign-reversing weight-preserving involution is constructed as follows. Take the maximal $i$ for which $\mu_{i-1} + 1 < \lambda_i$ (in our example, $i=3$). Such an $i$ exists because $\lambda/\mu$ is not a broken border strip. Choose a composition $\nu$ and permutation $\pi$ with $\pi(\nu + \mu + \delta) = \lambda + \delta$. Note that the maximality of $i$ implies $(\lambda + \delta)_{\pi(i-1)} = (\nu + \mu + \delta)_{i-1} \geq (\mu + \delta)_{i-1} > (\mu + \delta)_i \geq(\lambda + \delta)_{i+1}$ and so $\pi(i-1) \leq i$. Similarly, for $i \leq j < n$, we have $(\lambda + \delta)_{\pi(j)} = (\nu + \mu + \delta)_j \geq (\mu + \delta)_j \geq (\lambda + \delta)_{j+1}$ and $\pi(j) \leq j+1$. Choose the smallest $k \geq i$ with $\pi(k) \leq k$. Part of the permutation $\pi$ is
 $$\begin{pmatrix} \ldots & i-1 & i & i+1 & \ldots & k-1 & k & \ldots \\
 			  \ldots & \pi(i-1) & i+1 & i+2 & \ldots & k & \pi(k) & \ldots \end{pmatrix}.$$
 Note that $\pi(i-1), \pi(k) \leq i$. In the example, we have $k=3,4,3,4,3,4,3,4$, respectively.\\
 Define $\varphi(\pi) = \pi \cdot (i-1,k)$. Then $\varphi(\pi)$ has the following form:
 $$\begin{pmatrix} \ldots & i-1 & i & i+1 & \ldots & k-1 & k & \ldots \\
  			  \ldots & \pi(k) & i+1 & i+2 & \ldots & k & \pi(i-1) & \ldots \end{pmatrix}.$$
 Clearly, $\varphi$ is a sign-reversing involution. Furthermore,
 $$\nu_{i-1} = (\nu + \mu + \delta)_{i-1} - (\mu + \delta)_{i-1} = (\lambda + \delta)_{\pi(i-1)} - (\mu + \delta)_{i-1} \geq (\lambda + \delta)_i - (\mu + \delta)_{i-1} > 0$$
 and
 $$\nu_{k} = (\nu + \mu + \delta)_{k} - (\mu + \delta)_{k} = (\lambda + \delta)_{\pi(k)} - (\mu + \delta)_{k} > (\lambda + \delta)_i - (\mu + \delta)_{i-1} > 0.$$
 These are the only entries that change in $\nu$ when we take $\varphi(\pi)$ instead of $\pi$, and since they are both strictly positive before and after the change, $\varphi$ preserves weight.\\
 This completes the proof.
\end{proof}

\begin{proof}[Proof of Theorem \ref{thm5}]
 By equation \eqref{intro5},
 $$\chi_\lambda (T_w) = \sum_{\nu} K_{\nu,\lambda}^{-1} \eta_\nu (T_w).$$
 Now note that Theorem \ref{a1} says that
 $$\eta_\nu (T_w)= [m_\nu] \overline p_\mu,$$
 where $\mu$ is the composition corresponding to $w$. Therefore
 $$\chi_\lambda (T_w) = \sum_{\nu} K_{\nu,\lambda}^{-1} [m_\nu] \overline p_\mu = \sum_{\nu} K_{\nu,\lambda}^{-1} \langle \overline p_\mu, h_\nu \rangle = \left\langle \overline p_\mu, \sum_{\nu} K_{\nu,\lambda}^{-1} h_\nu \right\rangle = \langle \overline p_\mu, s_\lambda \rangle,$$
 where $\langle \cdot,\cdot \rangle$ is the standard scalar product defined by $\langle h_\lambda, m_\mu \rangle = \delta_{\lambda,\mu}$. The result follows by Lemma \ref{a5} and induction on the length of $\mu$.
\end{proof}


Theorem \ref{thm5} implies the following interesting fact.

\begin{cor}
 The endomorphism (and involution, see \cite[\S 7.6]{StanEC}) $\omega$ defined on the algebra of symmetric functions by $\omega(e_r) = h_r$ satisfies $\omega(\overline p_\mu) = \widetilde p_\mu$.
\end{cor}
\begin{proof}
 We use the facts that $\omega(s_\lambda) = s_{\lambda'}$ and that $\omega$ is an isometry with respect to $\langle \cdot,\cdot \rangle$. It is enough to prove that $\omega(\overline p_r) = \widetilde p_r$ for all $r$. We have
 $$\langle \omega(\widetilde p_r),s_\lambda \rangle = \langle \widetilde p_r,s_{\lambda'} \rangle = \langle \widetilde p_r, \sum_{\nu} K_{\nu,\lambda'}^{-1} h_\nu \rangle= \sum_\nu K_{\nu,\lambda'}^{-1} [m_\nu] \widetilde p_r = $$
 $$= \sum_\nu K_{\nu,\lambda'}^{-1} \epsilon_\nu(T_{\gamma_r}) = \chi_\lambda(T_{\gamma_r}) = \langle \overline p_r, s_\lambda \rangle$$
 and the claim follows.
\end{proof}

\section{Characters in type B and D} \label{bd}

Every signed composition $\lambda = (\lambda_1,\ldots,\lambda_p) \vdash n$ has a corresponding subgroup of $\f S_{-n}$ which is naturally isomorphic to $\f S_{\lambda_1} \times \f S_{\lambda_2} \times \cdots \times \f S_{\lambda_p}$; call this subgroup the \emph{quasi-parabolic subgroup corresponding to $\lambda$}, and denote it by $\f S_\lambda^B$. A quasi-parabolic subgroup is parabolic if and only if $\lambda_k > 0$ for $k \geq 2$. 

\medskip

For a signed composition $\lambda$, we have
$$\# \f S_\lambda^B = 2^{\sum_{\lambda_i<0} |\lambda_i|} \prod |\lambda_i|!,$$
and the index of $\f S_\lambda^B$ is
$$2^{\sum_{\lambda_i>0} \lambda_i} \binom{n}{|\lambda_1|,\ldots,|\lambda_p|}.$$

\medskip

The length of a signed permutation $w=w_1 \cdots w_n$ is equal to
\begin{equation} \label{bd3}
 \#\set{i \colon w_i < 0} + \#\set{i < j \colon |w_i|>|w_j|} + 2 \cdot \#\set{i < j \colon w_j<0,|w_i|<|w_j|}.
\end{equation}
Signed permutations $w$ and $v$ are in the same left coset of $\f S_\lambda^B$ if and only if for each $k = 1,\ldots,p$, we have either
\begin{itemize}
 \item $\lambda_k > 0$ and the sets 
 $$\set{w_{|\lambda_1| + \ldots + |\lambda_{k-1}| + 1},w_{|\lambda_1| + \ldots + |\lambda_{k-1}| + 2},\ldots,w_{|\lambda_1| + \ldots + |\lambda_{k}|}}$$
 and 
 $$\set{v_{|\lambda_1| + \ldots + |\lambda_{k-1}| + 1},v_{|\lambda_1| + \ldots + |\lambda_{k-1}| + 2},\ldots,v_{|\lambda_1| + \ldots + |\lambda_{k}|}}$$
 are equal, or
 \item $\lambda_k < 0$ and the sets 
 $$\set{|w_{|\lambda_1| + \ldots + |\lambda_{k-1}| + 1}|,|w_{|\lambda_1| + \ldots + |\lambda_{k-1}| + 2}|,\ldots,|w_{|\lambda_1| + \ldots + |\lambda_{k}|}|}$$
 and 
 $$\set{|v_{|\lambda_1| + \ldots + |\lambda_{k-1}| + 1}|,|v_{|\lambda_1| + \ldots + |\lambda_{k-1}| + 2}|,\ldots,|v_{|\lambda_1| + \ldots + |\lambda_{k}|}|}$$
 are equal. 
\end{itemize}

\medskip
 
Let us denote by $X_\lambda^B$ (or $X_J^B$, if $\lambda$ comes from a subset $J \subseteq S_n^B$) the set of coset representatives of $\f S_\lambda^B$ of minimal length. A signed permutation $x$ is in $X_\lambda^B$ if and only if for each $k = 1,\ldots,p$, we have either
\begin{itemize}
 \item $\lambda_k > 0$ and 
 $$x_{|\lambda_1| + \ldots + |\lambda_{k-1}| + 1} < x_{|\lambda_1| + \ldots + |\lambda_{k-1}| + 2} < \ldots < x_{|\lambda_1| + \ldots + |\lambda_{k}|},$$
 or
 \item $\lambda_k < 0$ and
 $$0 < x_{|\lambda_1| + \ldots + |\lambda_{k-1}| + 1} < x_{|\lambda_1| + \ldots + |\lambda_{k-1}| + 2} < \ldots < x_{|\lambda_1| + \ldots + |\lambda_{k}|}.$$
\end{itemize}

For a subset $J \subseteq S_n^B$ and the corresponding signed composition $\lambda = (\lambda_1,\ldots,\lambda_p) \vdash n$, there is a natural bijection $\Phi_J^B$ between left cosets of $\f S_J^B$ and integer sequences $a = a_1 a_2 \cdots a_n$ satisfying
\begin{enumerate}
 \item $1 \leq |a_i| \leq p$ for $i = 1,\ldots, n$,
 \item $\#\set{i \colon |a_i| = k} = |\lambda_k|$ for $k = 1,\ldots,p$,
 \item if $t \in J$, then $a_i \neq -1$ for all $i = 1,\ldots,n$.
\end{enumerate}
Fix a coset $v \f S_J^B$. If $t \in J$ and
$$|v^{-1}(i)| \leq |\lambda_1|$$
take $a_i = 1$. Otherwise, take $a_i = k$ if
$$|\lambda_1|+\ldots+|\lambda_{k-1}| + 1 \leq v^{-1}(i) \leq |\lambda_1|+\ldots+|\lambda_{k}|,$$
and $a_i = - k$ if
$$|\lambda_1|+\ldots+|\lambda_{k-1}| + 1 \leq - v^{-1}(i) \leq |\lambda_1|+\ldots+|\lambda_{k}|.$$
\medskip

Again, the map is well-defined, and it is also easy to see that it is a bijection. By slight abuse of notation, we also use $\Phi_J$ as the corresponding bijection between $X_J$ and integer sequences with properties (1), (2) and (3).

\begin{exm}
 Take $n=8$, $J=\set{s_1,s_2,s_4,s_6,s_7}$ (hence $\lambda=323$) and $v = 4\underline{1}573\underline{2}\underline{8}6$. Since
 $$1 \leq  v^{-1}(4),-v^{-1}(1),v^{-1}(5) \leq 3 = |\lambda_1|,$$
 $$|\lambda_1|+1 = 4\leq v^{-1}(7),v^{-1}(3) \leq 5 = |\lambda_1|+|\lambda_2|, $$
 $$|\lambda_1|+|\lambda_2|+1 = 6 \leq -v^{-1}(2),-v^{-1}(8),v^{-1}(6) \leq 8,$$
 the corresponding coset $v \f S_J^B$ maps to $\underline{1}\underline{3}21132\underline{3}$ with $\Phi_J^B$, and the sequence $3121\underline{3}32\underline{1}$ is the image of the coset that contains $\underline{8}2437\underline{5}61$.\\
 For $J=\set{t,s_1,s_2,s_4,s_6,s_7}$ (hence $\lambda=\underline{3}23$) and $v = 4\underline{1}573\underline{2}\underline{8}6$, on the other hand, the coset $v \f S_J^B$ maps to $1\underline{3}21132\underline{3}$, and the sequence $3121\underline{3}321$ is the image of the coset that contains $82437\underline{5}61$.
\end{exm}

We can describe the bijection as follows. Pick a coset representative $v$, write it in one-line notation, and draw bars in positions given by $\lambda$. In the examples above, that would be $4\underline{1}5|73|\underline{2}\underline{8}6$. The $i$-th element of the $\Phi_J(v)$ tells us in which of the ``slots'' $\pm i$ is located, and with what sign. For example, $-2$ is located in the third slot, so $a_2=-3$. Furthermore, if $t \in J$ (i.e. if $\lambda_1<0$), then change each occurrence of $-1$ to $1$. The inverse is similarly intuitive. Write the locations (positive and negative) of $1$'s and $-1$'s in $a=a_1\cdots a_n$ in increasing order, then the locations of $2$'s and $-2$'s in increasing order, etc. The resulting permutation is a minimal length representative of the coset $\Phi_J^{-1}(a)$. For the example $3121\underline{3}32\underline{1}$ above, this minimal coset representative is $\underline{8}2437\underline{5}16$.

\medskip

Again, we have the following important lemma.

\begin{lemma} \label{bd1}
 Let $\sim_J$ denote the relation of being in the same coset of the subgroup $\f S_J^B$. For $s \in S_n^B$, and every $x,v \in \f S_{-n}$, we have
 $$vsx \sim_J x \not\sim_J s x \Longrightarrow s \in \rw(v).$$
\end{lemma}
\begin{proof}
 Suppose that $s = s_i \notin \rw(v)$ and that $x \not\sim_J s_i x$. The latter means that 
 $$|\lambda_1|+\ldots+|\lambda_{k-1}| + 1 \leq |x^{-1}(i)| \leq |\lambda_1|+\ldots+|\lambda_{k}|$$
 and
 $$|\lambda_1|+\ldots+|\lambda_{j-1}| + 1 \leq |x^{-1}(i+1)| \leq |\lambda_1|+\ldots+|\lambda_{j}|,$$
 where $k \neq j$. From here, the proof is almost exactly the same as in type A. On the other hand, take $s = t \notin \rw(v)$ with $x \not\sim_J t x$. If $x$ (in the one-line notation) contains $1$ (respectively $-1$), then $tx$ contains $-1$ (respectively $1$), and then also $vtx$ contains $-1$ (respectively, $1$). But then $vtx \not\sim_J x$.
\end{proof}

As in type A, the lemma implies that for $s \notin \rw(w'),\rw(w'')$ and $\rw(w') \cap \rw(w'') = \emptyset$,

\begin{equation} \label{bd2}
 \wP_{x,J} \left(\wT_{w'sw''} \cdot \wT_x\right) = \wP_{x,J}\left(\wT_{w'}\cdot \wP_{x,J}\left(\wT_s \cdot \wP_{x,J}\left(\wT_{w''} \cdot \wT_x\right)\right)\right).
\end{equation}

where
$$\wP_{x,J}(\wT_v) = \left\{ \begin{array}{ccl} \wT_v & : & v \sim_J x \\ 0 & : & \mbox{otherwise} \end{array} \right..$$

We can now prove Theorem \ref{thm2}.

\medskip

Assume that we are given a parabolic element $w$ of type $K$. For a transversal element $x \in X_J$, write $a = \Phi_J(x)$. We have one of the following cases.

\medskip

Suppose first that there is an $i$ such that $s_i$ appears in a reduced word for $w$, $w = w' s_i w''$, and such that $a_i < a_{i+1}$. There are several options. First, we can have $0 < k = a_i < a_{i+1}$. By definition, that means that
$$x^{-1}(i) \leq |\lambda_1| + \ldots + |\lambda_k| < x^{-1}(i+1).$$ 
Furthermore, every $x' \sim_J x$ also satisfies 
$$(x')^{-1}(i) \leq |\lambda_1| + \ldots + |\lambda_k| < (x')^{-1}(i+1),$$ 
i.e. $x'$ has $i$ strictly to the left of $i+1$. In particular, by \eqref{bd3}, $\ell(sx') > \ell(x')$ (which means $\wT_s \wT_{x'} = \wT_{sx'}$) and $sx' \not\sim_J x' \sim_J x$ (which means $\wP_{x,J}\left(\wT_{sx'}\right) = 0$). In other words, $\wP_{x,J}\left(\wT_s \cdot \wP_{x,J}\left(\wT_{w''} \cdot wT_x\right)\right) = 0$, and equation \eqref{bd2} implies that
$$\wP_{x,J} \left(\wT_{w} \wT_x\right) = 0.$$ 
We can also have $-k = a_i < 0 < a_{i+1}$. Then
$$-x^{-1}(i) \leq |\lambda_1| + \ldots + |\lambda_k| < x^{-1}(i+1),$$ 
and every $x' \sim_J x$ has $-i$ strictly to the left of $i+1$. Then $\ell(sx') > \ell(x')$ and, as before, $\wP_{x,J} \left(\wT_{w} \wT_x\right) = 0$. Lastly, we can have $-k = a_i < a_{i+1} < 0$. Then
every $x' \sim_J x$ has $-(i+1)$ strictly to the left of $-i$, and multiplying on the left by $s_i$ increases the length. We again conclude that
$$\wP_{x,J} \left(\wT_{w} \wT_x\right) = 0.$$ 

\medskip

Now suppose that $t$ appears in a reduced word for $w$, $w = w' t w''$, and that $a_1 \geq 2$ or that $a_1 =1$ and $t \notin J$. That means that in $x' \sim_J x$, $1$ appears. Then $\ell(t x') > \ell(x')$, and again we have
$$\wP_{x,J} \left(\wT_{w} \wT_x\right) = 0.$$ 

\medskip

On the other hand, suppose that none of the above holds, in other words, that $a$ satisfies all the conditions of Theorem \ref{thm2}. We prove by induction on $\ell(w) = \# K$ that in this case,
$$\wP_{x,J} \left(\wT_{w} \wT_x\right) = R^{d_K(a)+f_K(a)} \wT_{w_{a,J} x},$$
where $d_K(a)$ denotes the number of elements in the set $\set{i \colon s_i \in K, a_i > a_{i+1}}$, $f_K(a)$ is $1$ if $t \in K$ and $a_1<0$, and $0$ otherwise, $w_{a,J}$ is the subword of $w$ consisting of $s_i \in J$ with $a_i = a_{i+1}$ and $t$ if $a_1 > 0$, and $\ell(w_{a,J} x) = \ell(w_{x,J}) + \ell(x)$.

\medskip

Again, the statement is obvious for $K = \emptyset$, i.e.\hspace{-0.07cm} for $w = e$. Suppose that it holds for $w'$, and suppose that $w = sw'$ with $\ell(w) = \ell(w') + 1$. We have to evaluate
$$\wP_{x,J} \left(\wT_{w} \wT_x\right) = \wP_{x,J} \left(\wT_{s} \wT_{w'} \wT_x\right) = \wP_{x,J} \left( \wT_s \wP_{x,J}\left(\wT_{w'} \wT_x \right)\right) = R^{d_{K'}(a)+f_{K'}(a)} \wP_{x,J}\left( \wT_s \wT_{w_{a,J}' x}\right),$$
where $K' = K \setminus \set s$, $w_{a,J}'$ is the subword of $w'$ consisting of $s_i \in J$ with $a_i = a_{i+1}$ and $t$ if $a_1 > 0$. If $s = s_i$ and $a_i > a_{i+1}$, then in each of the three possible cases ($a_i>a_{i+1}>0$, $a_i>0>a_{i+1}$ and $0>a_i>a_{i+1}$) we have $\ell(s_i v) < \ell(v)$ for every $v \sim_J x$. That implies $\ell(s_i w_{x,J}' x) < \ell(w_{x,J}'x)$ and $\wP_{x,J}\left( \wT_s \wT_{w_{x,J}' x}\right) = R \wT_{w_{x,J}' x}$. We have $d_K(a) = d_{K'}(a)+1$, $f_K(a) = f_{K'}(a)$, $w_{x,J} = w_{x,J}'$ and
$$\wP_{x,J} \left(\wT_{w} \wT_x\right) = R^{d_{K'}(a)+1+f_{K'}(a)} \wT_{w_{x,J}' x} = R^{d_K(a)+f_K(a)} \wT_{w_{x,J} x}.$$
On the other hand, assume $s = s_i$ and $a_i = a_{i+1} = k > 0$. That means 
$\lambda_1+\ldots+\lambda_{k-1} + 1 \leq x^{-1}(i), x^{-1}(i+1) \leq \lambda_1+\ldots+\lambda_{k}$. Furthermore, because $x \in X_J$, we have $x^{-1}(i) + 1=  x^{-1}(i+1)$. As in type A, $\ell(s_i w_{x,J}' x) > \ell(w_{x,J}' x)$, $\wT_{s_i} \wT_{w_{x,J}' x} = \wT_{s_i w_{x,J}' x}$, $d_K(a)=d_{K'}(a)$, $f_K(a) = f_{K'}(a)$, $w_{x,J} = s_iw_{x,J}'$ and
$$\wP_{x,J} \left(\wT_{w} \wT_x\right) = R^{d_{K'}(a)+f_{K'}(a)} \wT_{s_iw_{x,J}' x} = R^{d_K(a)+f_K(a)} \wT_{w_{x,J} x}.$$
Also, $\ell(s_i w_{x,J}' x) = \ell(w_{x,J} x) = \ell(w_{x,J}' x) + 1 = \ell(w_{x,J}) + \ell(x)$. A similar reasoning applies if $s = s_i$ and $a_i = a_{i+1} = -k < 0$.

\medskip

Assume that $s = t$ and $a_1 < 0$. Then $w_{x,J}' x$ is $\ell(w_{x,J}'x)$, with one $-1$ changed to $1$. Therefore $\ell(t w_{x,J}' x) < \ell(w_{x,J}'x)$ and $\wP_{x,J}\left( \wT_t \wT_{w_{x,J}' x}\right) = R \wT_{w_{x,J}' x}$. Since $d_K(a)=d_{K'}(a)$, $f_K(a)=1$, $f_{K'}(a)=0$ and $w_{x,J} = w_{x,J}'$, this means that
$$\wP_{x,J} \left(\wT_{w} \wT_x\right) = R^{d_{K'}(a)+f_{K'}(a)+1} \wT_{w_{x,J}' x} = R^{d_K(a)+f_K(a)} \wT_{w_{x,J} x}.$$ 
Finally, assume that $s = t$ and $a_1=1>0$ (and then also $t \in J$). Then $\ell(s_i w_{x,J}' x) > \ell(w_{x,J}' x)$, $\wT_{s_i} \wT_{w_{x,J}' x} = \wT_{s_i w_{x,J}' x}$, $d_K(a)=d_{K'}(a)$, $f_K(a)=f_{K'}(a)$ and $w_{x,J} = t w_{x,J}'$, so
$$\wP_{x,J} \left(\wT_{w} \wT_x\right) = R^{d_{K'}(a)+f_{K'}(a)} \wT_{tw_{x,J}' x} = R^{d_K(a)+f_K(a)} \wT_{w_{x,J} x}.$$

A calculation analogous to the one in type A completes the proof in type B.

\medskip

The proof in type D is similar and we will leave it as an exercise for the reader. As is evident from the statement of Theorem \ref{thm3}, there are even more separate cases to check, but again, the basis of it all is a lemma analogue to Lemmas \ref{a1} and \ref{bd1}. Let us just mention that the length of a signed permutation $w=w_1 \cdots w_n$ in $\f S_{-n}^0$ is equal to
\begin{equation}
 \#\set{i < j \colon |w_i|>|w_j|} + 2 \cdot \#\set{i < j \colon w_j<0,|w_i|<|w_j|}.
\end{equation}

\section{Characters in type I} \label{i}

Theorem \ref{thm4} gives expressions for characters induced from parabolic subgroups, evaluated at any elements of the Hecke algebra, unlike in case A, where we get combinatorial descriptions of only minimal length representatives of conjugacy classes, and unlike in cases B and D, where we found combinatorial descriptions of minimal length representatives of only certain conjugacy classes. However, the proofs in this case are extremely technical. Every result is a careful study of several (similar) cases. For the sake of brevity, we omit proofs of certain cases in the proofs of the necessary lemmas.

\medskip

Let us start with the following useful formula, which can be easily proved by induction on $a$:
\begin{equation} \label{i1}
 \sum_{i=0}^a \binom{c+i}b = \binom{c+a+1}{b+1} - \binom c{b+1}.
\end{equation}

The next result provides a basis for all our computations in this Hecke algebra.

\begin{lemma} \label{i2}
 For $m$ even and $0 \leq k \leq m/2$, we have
 $$\wT_{(12)^k} \wT_{(12)^{m/2}} = \wT_{(21)^{m/2-k}} + \sum_{\ell(w)>m-2k} \sum_{j} {\textstyle \binom{\ell(w)-m+2k-j-1}{j}} R^{\ell(w)-m+2k-2j} \wT_w.$$
 For $m$ odd and $0 \leq k \leq (m-1)/2$, we have
 $$\wT_{(12)^k} \wT_{1(21)^{(m-1)/2}} = \wT_{2(12)^{(m-1)/2-k}} + \sum_{\ell(w)>m-2k}\sum_{j}  {\textstyle \binom{\ell(w)-m+2k-j-1}{j}} R^{\ell(w)-m+2k-2j} \wT_w.$$
\end{lemma}
\begin{skt}
 The first statement is obvious for $k=0$, assume it is true for $k < m/2$. We have to find the coefficient of $\wT_w$ in $\wT_{(12)^{k+1}} \wT_{(12)^{m/2}} = \wT_{12} \wT_{(12)^{k}} \wT_{(12)^{m/2}}$. Let us look first at the case when $\ell(w) > m-2k+2$. Multiplying $\wT_v$ by $\wT_{12}$ gives terms with length between $\ell(v)-2$ and $\ell(v)+2$. In particular, to find the coefficient at $\wT_w$ in 
 $$\wT_{12} \wT_{(12)^k} \wT_{(12)^{m/2}},$$
 it is enough to find the coefficient of $\wT_w$ in
 $$\wT_{12} \left( \sum_{|\ell(w)-\ell(v)| \leq 2} \sum_j \binom{\ell(v)-m+2k-j-1}{j} R^{\ell(v)-m+2k-2j} \wT_v\right).$$
 Assume that $w = (12)^{\ell/2}$. There are exactly three elements $v \in I_2(m)$ for which $[\wT_w] \wT_{12} \wT_v \neq 0$: $[\wT_w] \wT_{12} \wT_{(12)^{\ell/2-1}} = 1$, $[\wT_w] \wT_{12} \wT_{2(12)^{\ell/2-1}} = R$ and $[\wT_w] \wT_{12} \wT_{2(12)^{\ell/2}} = R$. In other words,
 $$[\wT_w] \wT_{(12)^{k+1}} \wT_{(12)^{m/2}} = \sum_j {\textstyle \binom{\ell-2-m+2k-j-1}j} R^{\ell-2-m+2k-2j} + $$
 $$+R \sum_j {\textstyle \binom{\ell-1-m+2k-j-1}j} R^{\ell-1-m+2k-2j} + R \sum_j {\textstyle \binom{\ell+1-m+2k-j-1}j} R^{\ell+1-m+2k-2j}= $$
 $$= \sum_j \left( {\textstyle \binom{\ell-m+2(k+1)-j-3}{j-2} + \binom{\ell-m+2(k+1)-j-3}{j-1} + \binom{\ell-m+2(k+1)-j-2}j} \right) R^{\ell-m+2(k+1)-2j} = $$
 $$= \sum_j \left( {\textstyle \binom{\ell-m+2(k+1)-j-2}{j-1} + \binom{\ell-m+2(k+1)-j-2}j} \right) R^{\ell-m+2(k+1)-2j}= $$
 $$= \sum_j \binom{\ell-m+2(k+1)-j-1}j  R^{\ell-m+2(k+1)-2j},$$
 which is the induction statement for $k+1$. Of course, we have to verify the same statement for $w=(21)^{\ell/2}$, $w = 1(21)^{(\ell-1)/2}$ and $w = 2(12)^{(\ell-1)/2}$, and we also have to verify the cases with $\ell(w) \leq m - 2k + 2$. For example, if $w = (21)^{m/2-k-1}$, the coefficient is equal to $[\wT_w] \wT_{12} \wT_{(21)^{m/2-k}} = 1$, if $w = 1(21)^{m/2-k-1}$, the coefficient is equal to
 $$[\wT_w] \wT_{12} \wT_{(21)^{m/2-k}} = [\wT_w] \wT_{12} \wT_{21} \wT_{(21)^{m/2-k-1}} = [\wT_w] (1 + R \wT_1 + R \wT_{121}) \wT_{(21)^{m/2-k-1}} = R,$$
 if $w = (12)^{m/2-k}$, then the coefficient is equal to
 $$[\wT_w] \wT_{12} (\wT_{(21)^{m/2-k}} + R \wT_{1(21)^{m/2-k}} + R \wT_{2(12)^{m/2-k}} + R^2 \wT_{(12)^{m/2-k+1}} + R^2 \wT_{(21)^{m/2-k+1}}) = $$
 $$= [\wT_w] R \wT_{12} \wT_{2(12)^{m/2-k}} = R^2,$$
 etc. We leave all other cases as exercises. The second equality is proved analogously.\qed
\end{skt}

\begin{lemma} \label{i3}
 If $w \in I_2(m)$ has even length $\ell \geq 2$ and $\ell(x) > 2(m-\ell)$, then
 $$[\wT_x] \wT_w \wT_x = \sum_{k} \binom{\ell-2m+2\ell(x)-1-k}{k} R^{\ell-2m+2\ell(x)-2k}.$$
 Now choose $w \in I_2(m)$ with odd length $\ell$. If $(\ell+1)/2 \leq \ell(x) < m - (\ell-1)/2$ and reduced expressions for $x$ and $w$ start with the same generator, we have $[\wT_x] \wT_w \wT_x = R$. Otherwise,
 $$[\wT_x] \wT_w \wT_x = \sum_{k} \binom{\ell-2m+2\ell(x)-1-k}{k} R^{\ell-2m+2\ell(x)-2k}.$$
\end{lemma}
\begin{skt}
 The proof consists of studying different cases and applying Lemma \ref{i2}. Suppose first that $w = (12)^a$, $x = (12)^b$, and $m$ is even. If $a + b \leq m/2$, then $\wT_w \wT_x = \wT_{(12)^{a+b}}$ and $[\wT_x] \wT_w \wT_x = 0$. If $a + b > m/2$, then 
 $$\wT_w \wT_x = \wT_{(12)^{a+b-m/2}} \wT_{(12)^{m/2}}.$$
 If $m-a-2b=0$, then, by Lemma \ref{i2}, the only term in $\wT_w \wT_x$ of length $2b$ is $\wT_{(21)^b}$; in other words, $[\wT_x] \wT_w \wT_x = 0$. If $a+2b-m > 0$, then the term of $\wT_x$ in $\wT_w \wT_x$ is, again by Lemma \ref{i2}, equal to
 $$\sum_{j} \binom{2b-m+2(a+b-m/2)-j-1}{j} R^{2b-m+2(a+b-m/2)-2j}= $$
 $$=\sum_{j} \binom{2a+4b-2m-j-1}{j} R^{2a+4b-2m-2j} \! = \!\sum_{j} \binom{\ell-2m+\ell(x)-j-1}{j} R^{\ell-2m+2\ell(x)-2j},$$
 as claimed.\\
 Suppose $w = (21)^a$, $x = (12)^b$, $b < m/2$, and $m$ is even. Note that $\wT_{21} \wT_{12} = \wT_2 (1 + R \wT_1) \wT_2 = 1 + R \wT_2 + R \wT_{212}$. That means that $\wT_{(21)^a} \wT_{(12)^b}$ equals
 $$\wT_{(21)^{a-1}} (1 + R \wT_2 + R \wT_{212})\wT_{(12)^{b-1}} = \wT_{(21)^{a-1}}\wT_{(12)^{b-1}} + R \wT_{2 (12)^{a-2}} \wT_{(12)^b} + R \wT_{2 (12)^{a-1}} \wT_{(12)^b}.$$
 Repeating this, we get that
 $$[\wT_{(12)^b}] \wT_{(21)^a}\wT_{(12)^b} = [\wT_{(12)^b}] R \left( \wT_2 \wT_{(12)^b} + R \wT_{2 (12)^{1}} \wT_{(12)^b} +\ldots + R \wT_{2 (12)^{a-1}} \wT_{(12)^b}\right),$$
 which is equal to 
 $$[\wT_{2(12)^b}] R \left(\wT_{(12)^b} + R \wT_{(12)^{1}} \wT_{(12)^b} +\ldots + R \wT_{(12)^{a-2}} \wT_{(12)^b} + R \wT_{(12)^{a-1}} \wT_{(12)^b}\right).$$
 We can ignore the terms $\wT_{(12)^{k}} \wT_{(12)^b}$ with $k < m-2b$, and we get
 $$R \sum_{k=m-2b}^{a-1} [\wT_{2(12)^b}] \wT_{(12)^{k+b-m/2}} \wT_{(12)^{m/2}} = $$
 $$= \sum_{k=m-2b}^{a-1} \sum_j \binom{2b+1-m+2(k+b-m/2)-j-1}j R^{2b+1-m+2(k+b-m/2)-2j+1}.$$
 The coefficient at $R^{\ell-2m+2\ell(x)-2i} = R^{2a-2m+4b-2i}$ in this expression is
 $$\sum_{k=m-2b}^{a-1} \binom{4b+a+k-2m-i-1}{k+i-a+1} = \sum_{k=0}^{a+2b-m-1} \binom{2b+a-m-1-i+k}{4b+2a-2m-2i-2} = $$
 $$=\binom{4b+2a-2m-1-i} {4b+2a-2m-2i-1} = \binom{4b+2a-2m-1-i}{i} = \binom{\ell-2m+2\ell(x)-i-1}{i},$$
 as claimed. Note that we used \eqref{i1} for the sum of binomials in the calculation. The rest of the computations are similar and we leave them as an exercise for the reader. \qed
\end{skt}

Recall that there is one element of length $m$ and $0$, and two elements of every other length. The lemma therefore implies that if $w$ has even length $\ell \geq 2$, then the coefficient of $R^{\ell - 2j}$ in
$$\sum_{x \in I_2(m)} [\wT_x] \wT_w \wT_x$$
is equal to
$$2 \sum_{i=m-j}^{m-1} \binom{\ell-m+i-1-j}{j+i-m} + \binom{\ell-1-j}j = 2 \sum_{i=0}^{j-1} \binom{\ell-2j-1+i}{\ell-2j-1} + \binom{\ell-1-j}j,$$
which is by \eqref{i1} equal to
$$2 \binom{\ell-j-1}{j-1} + \binom{\ell-j-1}j = \frac{\ell}{\ell-j} \binom{\ell - j}{j}.$$
Therefore
$$\eta_0(T_w) = q^{\ell/2} \sum_{j=0}^{\ell/2-1} \frac{\ell}{\ell-j} \binom{\ell-j}{j} (q^{1/2}-q^{-1/})^{\ell-2j} = \sum_{j=0}^{\ell/2-1} \frac{\ell}{\ell-j} \binom{\ell-j}{j} q^j (q-1)^{\ell-2j}.$$
The calculation for odd length $\ell$ and the coefficient of $R^{\ell - 2j}$ is completely analogous for $j=0,\ldots,(\ell-3)/2$, and the coefficient of $R^1$ in 
$$\sum_{x \in I_2(m)} [\wT_x] \wT_w \wT_x$$
is equal to
$$\left(m-\frac{\ell-1}2 - \frac{\ell+1}2 \right) \cdot 1 + \left(m - \left(m - \frac{\ell-1}2\right)\right) \cdot 2 + 1 = m.$$
Therefore 
$$\eta_0(T_w) = q^{\ell/2} \left(mR + \sum_{j=0}^{(\ell-3)/2} \frac{\ell}{\ell-j} \binom{\ell-j}{j} (q^{1/2}-q^{-1/})^{\ell-2j}\right) = $$
$$= m q^{(\ell+1)/2} - m q^{(\ell-1)/2} + \sum_{j=0}^{(\ell-3)/2} \frac{\ell}{\ell-j} \binom{\ell-j}{j} q^j (q-1)^{\ell-2j}.$$
 
The sign character is identical to the trivial character on the trivial subalgebra, so $\epsilon_0 = \eta_0$.

\medskip

For induction from the subalgebra $H_1$, note first that $X_1^I$, the set of minimal length coset representatives of the subgroup $\set{e,1}$, consists of $e$ and all $x$ whose reduced word representation ends with $2$, except for the longest element.

\begin{lemma} \label{i4}
 If $w \in I_2(m)$ has even length $\ell \geq 2$ and $x \in X_1^I$, then
 $$[\wT_{x1}] \wT_w \wT_x = \sum_{k} \binom{\ell-2m+2\ell(x)-k}{k} R^{\ell-2m+2\ell(x)+1-2k}.$$
 Now choose $w \in I_2(m)$ with odd length $\ell$. If $m$ is even, $w \sim 1$ and $\ell(x) \in \set{(\ell-1)/2,m-(\ell+1)/2}$, or $m$ is odd, $w \sim 1$ and $\ell(x) = (\ell-1)/2$, or $m$ is odd, $w \sim 2$ and $\ell(x) = m - (\ell+1)/2$, then $[\wT_{x1}] \wT_w \wT_x = 1$. Otherwise,
 $$[\wT_{x1}] \wT_w \wT_x = \sum_{k} \binom{\ell-2m+2\ell(x)-k}{k} R^{\ell-2m+2\ell(x)+1-2k}.\eqno \qed$$
\end{lemma}

The proof of this lemma is almost identical to the proof of Lemma \ref{i3}.

\medskip

Now we can finish the proof of the theorem.

\medskip

First suppose that $\ell(w)$ is even. If $w = e$, we clearly have $\wT_e \wT_x = \wT_x$ for every $x$ in the transversal, so $\eta(e) = \epsilon(e) = m$. For $\ell(w) \geq 2$, we have
$$\eta_1(\wT_w) = \sum_{x \in X_1^I} \left([\wT_{x}] \wT_w \wT_x + q^{1/2} [\wT_{x1}] \wT_w \wT_x \right) = $$
$$ = \sum_{i=0}^{m-1}  \left( \sum_{k}{\textstyle \binom{\ell-2m+2i-1-k}k} R^{\ell-2m+2i-2k}  + q^{1/2}\sum_k \textstyle{\binom{\ell-2m+2i-k}{k}}R^{\ell-2m+2i+1-2k}\right)$$
by Lemmas \ref{i3} and \ref{i4}. The coefficient at $R^{2j}$ for $j \geq 2$ is
$$\sum_{i=0}^{m-1} \binom{\ell-2m+2i-1-(\ell/2-m+i-j)}{\ell/2-m+i-j} = \sum_{i=0}^{m-1} \binom{\ell/2-m+i+j-1}{2j-1},$$
which is $\binom{\ell/2+j-1}{2j}$ by \eqref{i1}. Similarly, the coefficient of $R^{2j-1}$ for $j \geq 1$ is
$$\sum_{i=0}^{m-1} q^{1/2} \binom{\ell-2m+2i-(\ell/2-m+i+1-j)}{\ell/2-m+i+1-j} = q^{1/2} \sum_{i=0}^{m-1}  \binom{\ell/2-m+i+j-1}{2j-2},$$
which is equal to $q^{1/2} \binom{\ell/2+j-1}{2j-1}$ by \eqref{i1}. To finish the proof in this case, use $\eta_1(T_w) = q^{\ell/2} \eta_1(\wT_w)$. The proof for $\epsilon_1$, $\eta_2$ and $\epsilon_2$ is analogous.

\medskip

Take $\ell(w)$ odd, $m$ odd. Again, 
$$\eta_1(\wT_w) = \sum_{x \in X_1^I} \left([\wT_{x}] \wT_w \wT_x + q^{1/2} [\wT_{x1}] \wT_w \wT_x \right).$$
Note that there are $m-\ell$ elements of $X_1^I$ satisfying $(\ell+1)/2 \leq \ell(x) < m - (\ell-1)/2$, half of which start with the same generator as $w$. Furthermore, there is exactly one $x \in X_1^I$ with $\ell(x) < m - (\ell-1)/2$ and $[\wT_{x1}] \wT_w \wT_x = 1$. Therefore $\eta_1(\wT_w)$ is equal to 
$$\sum_{i=0}^{m-1}  \left( \sum_{k} {\textstyle \binom{\ell-2m+2i-1-k}{k}} R^{\ell-2m+2i-2k} + q^{1/2}\sum_{k} \textstyle{\binom{\ell-2m+2i-k}{k}} R^{\ell-2m+2i+1-2k}\right) + q^{1/2} + \frac{m-\ell}{2}R.$$
The coefficient at $R^{2j}$ for $j \geq 2$ is
$$\sum_{i=0}^{m-1} q^{1/2} \binom{\ell-2m+2i-((\ell+1)/2-m+i-j)}{(\ell+1)/2-m+i-j} = q^{1/2}\sum_{i=0}^{m-1}\binom{(\ell-1)/2-m+i+j}{2j-1},$$
which is $q^{1/2} \binom{(\ell-1)/2+j}{2j}$. Similarly, the coefficient at $R^{2j-1}$ for $j \geq 2$ is
$$\sum_{i=0}^{m-1}\binom{\ell-2m+2i-1-((\ell+1)/2-m+i-j)}{(\ell+1)/2-m+i-j} = \sum_{i=0}^{m-1}\binom{(\ell-1)/2-m+i+j-1}{2j-2},$$
which is equal to $\binom{(\ell-1)/2+j-1}{2j-1}$. The coefficient at $R^0$ is $q^{1/2}$, and the coefficient at $R^1$ is
$$\frac{m-\ell}{2} + \sum_{i=0}^{m-1}\binom{\ell-2m+2i-1-((\ell+1)/2-m+i-1)}{(\ell+1)/2-m+i-1} = \frac{m-\ell}{2} + \frac{\ell-1}{2} = \frac{m-1}2$$
This, together with $\eta_1(T_w) = q^{\ell/2} \eta_1(\wT_w)$ and analogous computations for $\epsilon_1$, $\eta_2$ and $\epsilon_2$, finishes the proof.

\medskip

Suppose that $\ell(w)$ is odd, $m$ is even, $w \sim 1$. Again, there are $m-\ell$ elements of $X_1^I$ satisfying $(\ell+1)/2 \leq \ell(x) < m - (\ell-1)/2$, half of which start with the same generator as $w$. Furthermore, there are exactly two elements $X_1^I$ satisfying $[\wT_{x1}] \wT_w \wT_x = 1$. A similar calculation to the one above shows that the coefficient of $R^{2j}$ in $\eta_1(\wT_w)$ for $j \geq 2$ is $q^{1/2} \binom{(\ell-1)/2+j}{2j}$, the coefficient at $R^{2j-1}$ for $j \geq 2$ is $\binom{(\ell-1)/2+j-1}{2j-1}$, the coefficient at $R^0$ is $2q^{1/2}$, and the coefficient at $R^1$ is $(m-1)/2$. This, together with $\eta_1(T_w) = q^{\ell/2} \eta_1(\wT_w)$ and analogous computations for $\epsilon_1$, $\eta_2$ and $\epsilon_2$, finishes the proof.

\medskip

Finally, choose $\ell(w)$ odd, $m$ even, $w \sim 2$. All the coefficients are the same as in the previous case, except the one at $R^0$, which is $0$.

\section{Concluding remarks} \label{remarks}

We gave a complete description of \emph{all} values of characters of representations induced from \emph{all} parabolic subalgebras of $H_m^I$. In this sense, our results in case I are complete (note, though, that there are not enough parabolic subalgebras to use results on characters induced from them to describe all irreducible characters of $H_m^I$).

\medskip

Theorems \ref{thm1} and \ref{thm5} give combinatorial descriptions of $\eta_\lambda(T_w)$ and $\epsilon_\lambda(T_w)$ only for parabolic elements $w$. This is not a serious flaw, since we can use Theorem \ref{intro7} to find the remaining values. Indeed, if $w(i) \leq i+1$ for all $i$, then $w$ is parabolic and we can use Theorem \ref{thm1} or \ref{thm5}. Otherwise, take the smallest $i$ such that $w(i) = j+1 > i+1$. Then $\ell(s_j w) < \ell(w)$ and $\ell(s_j w s_j) \leq \ell(w)$; by Theorem \ref{intro7}, $\chi(T_w)$ can be expressed in terms of $\chi(T_{s_jw})$ and $\chi(T_{s_jws_j})$, and we can find the values for every $w$ by induction on $\ell(w)$.

\medskip

On the other hand, our results for types B and D are not complete, since in $B_n$ and $D_n$, not every element is conjugate to a parabolic element. For example, any conjugate of $t s_1 t s_1 = (1)^- (2)^-$ contains at least two copies of $t$ in any reduced expression, and is therefore not parabolic. It would be very interesting to find an extension of Theorems \ref{thm2} and \ref{thm3} that would deal with such elements. The result will still be a sum over integer sequences described in these theorems, but we will \emph{not} have a weight of the form $q^a (q-1)^b$ for all sequences.

\medskip

It is, however, possible to extend the theorems in cases B and D to all signed compositions $\lambda$. The formulas get more complicated, and we only state the case B and do not give a proof. 

\medskip

Recall that every signed composition of $n$ has a corresponding subgroup of $\f S_{-n}$ which is naturally isomorphic to $\f S_{\lambda_1} \times \f S_{\lambda_2} \times \cdots \times \f S_{\lambda_p}$; call this subgroup the \emph{quasi-parabolic subgroup corresponding to $\lambda$}, and denote it by $\f S_\lambda^B$.

\begin{thm}
 Given a signed composition $\lambda = (\lambda_1,\ldots,\lambda_p) \vdash n$, denote by $\eta_\lambda$ (respectively, $\epsilon_\lambda$) the character of the representation of $H_n^B$ induced from the trivial (respectively, sign) representation of the quasi-parabolic subalgebra $H_\lambda$. For a parabolic element $w$ of $\f S_{-n}$ of type $K$, we have
 $$\eta_\lambda(\wT_{w}) = \sum_a q^{e_K(a) + g_K(a)} (q-1)^{d_K(a) + f_K(a)}$$
 and
 $$\epsilon_\lambda(\wT_{w}) = \sum_a (-1)^{e_K(a) + g_K(a)} (q-1)^{d_K(a) + f_K(a)},$$
 where the sums are over all integer sequences $a = a_1 a_2 \cdots a_n$ satisfying
 \begin{enumerate}
  \item $1 \leq |a_i| \leq p$ for $i = 1,\ldots, n$,
  \item $\#\set{i \colon |a_i| = k} = |\lambda_k|$ for $k = 1,\ldots,p$,
  \item if $\lambda_k > 0$, then $a_i \neq -i$ for all $i = 1,\ldots,n$,
  \item if $s_i \in K$, then $a_i \geq a_{i+1}$, for $i = 1,\ldots,n-1$,
  \item if $t \in K$, then either $a_1<0$, or $a_1>0$ and $\lambda_{a_1} < 0$.
 \end{enumerate}
 and where
 \begin{itemize}
  \item $d_K(a)$ is the number of elements in the set $\set{i \colon s_i \in K, a_i > a_{i+1}}$,
  \item $e_K(a)$ is the number of elements in the set $\set{i \colon s_i \in K, a_i = a_{i+1}}$,
  \item $f_K(a)$ is $1$ if $t \in K$ and $a_1 < 0$, and $0$ otherwise.
  \item $g_K(a)$ is $1 + 2 (|\lambda_1| + \ldots + |\lambda_{a_1-1}|)$ if $t \in K$ and $a_1>0$, and $0$ otherwise,
 \end{itemize}
\end{thm}

Furthermore, it would be interesting to extend the analysis of the second part of Section \ref{a} to types B and D. A quantum Murnaghan-Nakayama rule of type B and D (or something similar) would be needed for this purpose.

\medskip

A more ambitious project would find combinatorial descriptions for a Hecke algebra belonging to \emph{any (finite irreducible) Coxeter group}. The analogue of Lemma \ref{a1} should be true in general. However, it is far from clear what a common combinatorial description of Theorems \ref{thm1}, \ref{thm2}, \ref{thm3} and \ref{thm4} would be.

\medskip

In the preparation of this paper, an error was found in \cite[\S 8]{RamFrob}. The entry for $\lambda=321$ and $\mu=222$ should be $2q^3 - 6q^2 + 6q - 2$, not $q^3 - 5q^2 + 5q - 1$.

\section*{Acknowledgments}

This paper is a byproduct of a collaboration with Mark Skandera \cite{konska}. The author is grateful for helpful suggestions and comments. Special thanks also go to Arun Ram.


\begin{thebibliography}{10}

\bibitem[BB05]{bb}
{\sc A. Björner and F. Brenti}.
\newblock {\em Combinatorics of Coxeter groups}, Graduate Texts in Mathematics, vol. 231, Springer, New York (2005)

\bibitem[GP00]{gp}
{\sc M.~Geck and G.~Pfeiffer}.
\newblock {\em Characters of finite {C}oxeter groups and {I}wahori-{H}ecke
  algebras\/}, vol.~21 of {\em London Mathematical Society Monographs. New
  Series\/}.
\newblock The Clarendon Press Oxford University Press, New York (2000).

\bibitem[KS]{konska}
{\sc M.~Konvalinka and M.~Skandera}.
\newblock A quantization of a theorem of Goulden and Jackson.
\newblock preprint (2008)

\bibitem[Ram91]{RamFrob}
{\sc A.~Ram}.
\newblock A {F}robenius formula for the characters of the {H}ecke algebras.
\newblock {\em Invent. Math.\/}, {\bf 106}, 3 (1991) pp. 461--488.

\bibitem[RR97]{RamRemmel}
{\sc A.~Ram, J.~Remmel}.
\newblock Applications of the Frobenius formulas for the characters of the symmetric group and the Hecke algebras of type $A$.
\newblock {\em J. Algebraic Combin.\/}, {\bf 6 } (1997),  no. 1, pp. 59--87

\bibitem[RRW96]{RRW}
{\sc A.~Ram, J.~Remmel, T.~Whitehead}.
\newblock Combinatorics of the $q$-basis of symmetric functions.
\newblock {\em J. Combin. Theory Ser. A\/}, {\bf 76}  (1996),  no. 2, pp. 231--271

\bibitem[Sta99]{StanEC} R. P. Stanley, \emph{Enumerative combinatorics}, Vol. 2, Cambridge University Press, Cambridge, 1999
  
\end{thebibliography}
\end{document}